\documentclass[a4paper,12pt]{article}

\usepackage{psfrag}
\usepackage[leqno]{amsmath}
\usepackage{graphics}
\usepackage{epsfig}
\usepackage[mathscr]{euscript}
\usepackage{mathrsfs}
\usepackage{amsmath,amsthm,amsfonts}
\usepackage{amssymb}
\usepackage{amscd}
\usepackage{enumerate}
\usepackage[dvips]{color}

\theoremstyle{plain}
\newtheorem{lemma}{Lemma}[section]
\newtheorem{theorem}[lemma]{Theorem}
\newtheorem{prop}[lemma]{Proposition}
\newtheorem{cor}[lemma]{Corollary}

\theoremstyle{definition}
\newtheorem{defi}[lemma]{Definition}
\newtheorem{remark}[lemma]{Remark}
\newtheorem{exa}[lemma]{Example}

\newcommand{\Ho}{\mathfrak{H}}
\newcommand{\s}{\mathcal{S}}

\title{Fat Hoffman graphs with smallest eigenvalue
at least $-1-\tau$}

\author{
{\sc Akihiro MUNEMASA}\\
[1ex]
{\small 
Graduate School of Information Sciences,} \\
{\small Tohoku University, 
Sendai 980-8579, Japan} \\
{\small 
{\it E-mail address}: {\tt munemasa@math.is.tohoku.ac.jp}}\\
\\
{\sc Yoshio SANO}\\
[1ex]
{\small
Division of Information Engineering} \\
{\small
Faculty of Engineering, Information and Systems} \\
{\small University of Tsukuba,
Ibaraki 305-8573,
Japan} \\
{\small
{\it E-mail address}: {\tt sano@cs.tsukuba.ac.jp}}\\
\\
{\sc Tetsuji TANIGUCHI}\\
[1ex]
{\small 
Matsue College of Technology, 
Matsue 690-8518, Japan} \\
{\small 
{\it E-mail address}: {\tt tetsuzit@matsue-ct.ac.jp}}
}

\date{}

\begin{document}

\maketitle

\begin{abstract}
In this paper, we show that all fat Hoffman graphs 
with smallest eigenvalue at least 
$-1-\tau$, where $\tau$ is the golden ratio,
can be described by a finite set of fat 
$(-1-\tau)$-ir\-re\-duc\-i\-ble Hoffman graphs.
In the terminology of Woo and Neumaier, we mean
that every fat Hoffman graph with smallest eigenvalue at least 
$-1-\tau$ is an $\mathcal{H}$-line graph, where $\mathcal{H}$
is the set of isomorphism classes of maximal fat 
$(-1-\tau)$-ir\-re\-duc\-i\-ble Hoffman graphs.
It turns out that there are 
$37$ fat $(-1-\tau)$-ir\-re\-duc\-i\-ble Hoffman graphs, up to isomorphism. 
\end{abstract}



\newpage

\section{Introduction}\label{sec:001}

P.~J.~Cameron, J.~M.~Goethals, J.~J.~Seidel, and 
E.~E.~Shult \cite{root} characterized graphs 
whose adjacency matrices have smallest eigenvalue at least $-2$ 
by using root systems.
Their results revealed that 
graphs with smallest eigenvalue 
at least $-2$ 
are generalized line graphs, except a finite number of graphs
represented by the root system $E_8$. 
Another characterization for generalized line graphs were given by 
D.~Cvetkovi\'{c}, M.~Doob, and S.~Simi\'{c} \cite{GLG} 
by determinig minimal forbidden subgraphs (see also \cite{new}). 
Note that graphs with smallest eigenvalue greater than $-2$ 
were studied by A.~J.~Hoffman \cite{hoffman0}. 

Hoffman \cite{hoffman1} also studied graphs 
whose adjacency matrices have 
smallest eigenvalue at least $-1-\sqrt{2}$ 
by using a technique of adding cliques to graphs. 
R.~Woo and A.~Neumaier \cite{hlg} formulated Hoffman's idea 
by introducing the notion of Hoffman graphs.  
A Hoffman graph is a simple graph with a distinguished independent set 
of vertices, called fat vertices, which can be considered as 
cliques of size infinity in a sense 
(see Definition~\ref{df:Hg}, and also \cite[Corollary 2.15]{JKMT}). 
To deal with graphs with bounded smallest eigenvalue, 
Woo and Neumaier introduced 
a generalization of line graphs 
by considering decompositions of Hoffman graphs. 
They gave a characterization of graphs 
with smallest eigenvalue at least $-1-\sqrt{2}$ 
in terms of Hoffman 
graphs by classifying
fat indecomposable
Hoffman graphs with smallest eigenvalue at least $-1-\sqrt{2}$.
This led them to prove a theorem which states that
every graph with smallest eigenvalue at least  $-1-\sqrt{2}$
and sufficiently large minimum degree is a subgraph of
a Hoffman graph admitting a decomposition into subgraphs
isomorphic to only four Hoffman graphs.
In the terminology of \cite{hlg}, this means that
every graph with smallest eigenvalue at least  $-1-\sqrt{2}$
and sufficiently large minimum degree is an $\mathcal{H}$-line
graph, where $\mathcal{H}$ is the set of four isomorphism 
classes of Hoffman graphs.
For further studies on 
graphs with smallest eigenvalue at least $-1-\sqrt{2}$, 
see the papers by T.~Taniguchi \cite{paperI,paperII} 
and by H.~Yu \cite{Yu}. 

Recently, 
H. J. Jang, J. Koolen, A. Munemasa, and T. Taniguchi 
\cite{JKMT} made the first step 
to classify the fat indecomposable Hoffman graphs 
with smallest eigenvalue $-3$. 
However, 
it seems that
there are so many such Hoffman graphs.
A key to solve this problem is 
the notion of {\it special graphs} 
introduced by Woo and Neumaier. 
A special graph is 
an edge-signed graph
defined for each Hoffman graph. 
Although non-isomorphic Hoffman graphs may have 
isomorphic special graphs, it is not difficult to
recover all the Hoffman graphs with a given special
graph in some cases.

In this paper, we 
introduce irreducibility of Hoffman graphs and 
classify fat 
$(-1-\tau)$-ir\-re\-duc\-i\-ble 
Hoffman graphs, where $\tau:=\frac{1+\sqrt{5}}{2}$ is the golden ratio.
This is a somewhat more restricted class of Hoffman graphs
than those considered in \cite{JKMT}, and
there are only $37$ such Hoffman graphs. 
As a consequence, 
every 
fat Hoffman graph with smallest eigenvalue 
at least $-1-\tau$
is a subgraph of 
a Hoffman graph admitting a decomposition into subgraphs
isomorphic to only $18$ Hoffman graphs.
In the terminology of \cite{hlg}, this means that
every fat Hoffman graph with smallest eigenvalue 
at least $-1-\tau$
is an $\mathcal{H}$-line
graph, where $\mathcal{H}$ is the set of $18$ isomorphism 
classes of maximal fat 
$(-1-\tau)$-ir\-re\-duc\-i\-ble
Hoffman graphs.

\section{Preliminaries}\label{sec:pre}
\subsection{Hoffman graphs and eigenvalues}

\begin{defi}\label{df:Hg}
A \emph{Hoffman graph} $\frak{H}$ is a pair 
$(H,\mu)$ of a graph $H$ 
and a vertex labeling $\mu:V(H) \to \{ {\bf slim}, {\bf fat} \}$ 
satisfying the following conditions:
(i) 
every vertex with label {\bf fat} is adjacent to at least
one vertex with label {\bf slim};
(ii) 
the vertices with label {\bf fat} are pairwise non-adjacent.

Let $V(\Ho):= V(H)$, $V^s(\Ho):= \mu^{-1}({\bf slim})$, 
$V^f(\Ho):= \mu^{-1}({\bf fat})$, 
and $E(\Ho):= E(H)$. 
We call a vertex 
in $V^s(\Ho)$ 
a \emph{slim vertex}, and
a vertex 
in $V^f(\Ho)$
a \emph{fat vertex} of $\Ho$. 
We represent a Hoffman graph $\Ho$ also by the triple 
$(V^s(\Ho), V^f(\Ho), E(\Ho))$. 

For a vertex $x$ of a Hoffman graph $\Ho$, 
we define $N^f_{\Ho}(x)$ (resp.\ $N^s_{\Ho}(x)$) 
to be the 
set of fat (resp.\ slim) neighbors of $x$ in $\Ho$.
The set of all 
neighbors of $x$ is denoted by $N_{\Ho}(x)$, 
that is, $N_{\Ho}(x) := N^f_{\Ho}(x) \cup N^s_{\Ho}(x)$. 

A Hoffman graph $\Ho$ is said to be \emph{fat} 
if every slim vertex of $\Ho$ has a fat neighbor. 
A Hoffman graph is said to be \emph{slim} 
if it has no fat vertex. 

Two Hoffman graphs $\Ho= (H, \mu)$ and $\Ho'= (H', \mu')$
	are said to be \emph{isomorphic} 
	if there exists an isomorphism from $H$ to $H'$ 
	which preserves the labeling.

A Hoffman graph $\Ho' = (H', \mu')$ is called
	an \emph{induced Hoffman subgraph} 
	(or simply a \emph{subgraph}) 
	of another Hoffman graph $\Ho= (H, \mu)$ 
	if $H'$ is an induced subgraph of $H$ and 
$\mu(x) = \mu'(x)$ holds for any vertex $x$ of $\Ho'$.

The subgraph of a Hoffman graph $\Ho$ induced by $V^s(\Ho)$ 
is called the \emph{slim subgraph} of $\Ho$. 
\end{defi}

\begin{defi}\label{df:eigen}
For a Hoffman graph $\frak{H}$, let $A$ be its adjacency matrix,
\[
A=
\begin{pmatrix}
A_s & C \\
C^T & O
\end{pmatrix}
\]
in a labeling in which the fat vertices come last.
The \emph{eigenvalues} of $\Ho$ are the eigenvalues of the
real symmetric matrix $B(\Ho):=A_s-CC^T$.
We denote by $\lambda_{\min}(\Ho)$ the smallest eigenvalue of $B(\Ho)$. 
\end{defi}

\begin{remark}
An ordinary graph $H$ without vertex labeling can be 
regarded as a slim Hoffman graph $\Ho$. 
Then the matrix $B(\Ho)$ 
coincides with the ordinary adjacency matrix of the graph $H$. 
Thus the eigenvalues of $H$ as a slim Hoffman graph 
are the same as the eigenvalues of $H$ 
as an ordinary graph in the usual sense. 
\end{remark}

\begin{exa}\label{ex:H-1-2-3}
Let $\Ho_{{\rm I}}$, $\Ho_{{\rm II}}$, and $\Ho_{{\rm III}}$ 
be the Hoffman graphs defined by 
\[
\begin{array}{lll}
V^s(\Ho_{{\rm I}}) = \{v_1\}, &
V^f(\Ho_{{\rm I}}) = \{f_1\}, &
E(\Ho_{{\rm I}}) = \{\{v_1,f_1\}\}, \\
V^s(\Ho_{{\rm II}}) = \{v_1\}, &
V^f(\Ho_{{\rm II}}) = \{f_1,f_2\}, &
E(\Ho_{{\rm II}}) = \{\{v_1,f_1\},\{v_1, f_2\}\}, \\
V^s(\Ho_{{\rm III}}) = \{v_1,v_2\}, &
V^f(\Ho_{{\rm III}}) = \{f_1\}, &
E(\Ho_{{\rm III}}) = \{\{v_1,f_1\}, \{v_2, f_1\}\} 
\end{array}
\]
(see Figure \ref{fig:Hoffman123}). 
Note that $\lambda_{\min}(\Ho_{{\rm I}})=-1$ and 
$\lambda_{\min}(\Ho_{{\rm II}}) = \lambda_{\min}(\Ho_{{\rm III}})=-2$. 
\end{exa}

\begin{figure}[!h]
\begin{center}
\begin{tabular}{ccccc}
    \includegraphics[scale=0.15]{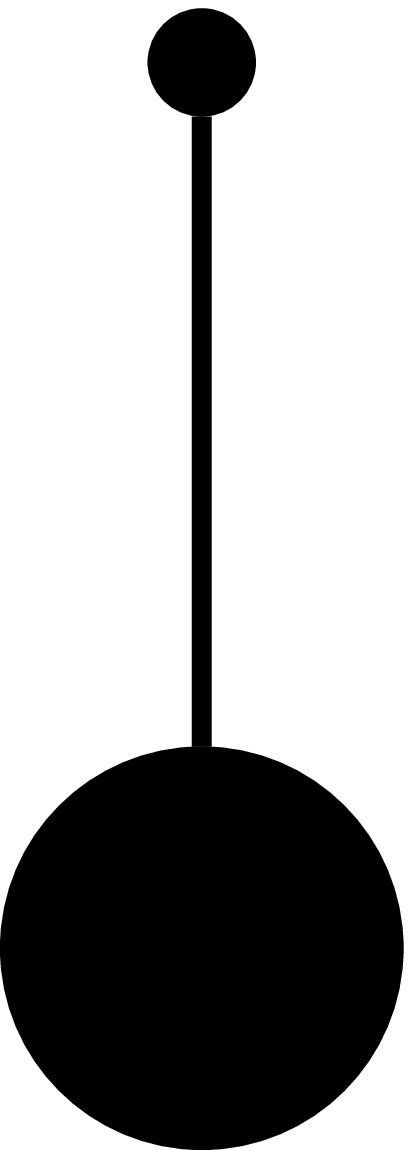} & &
    \includegraphics[scale=0.15]{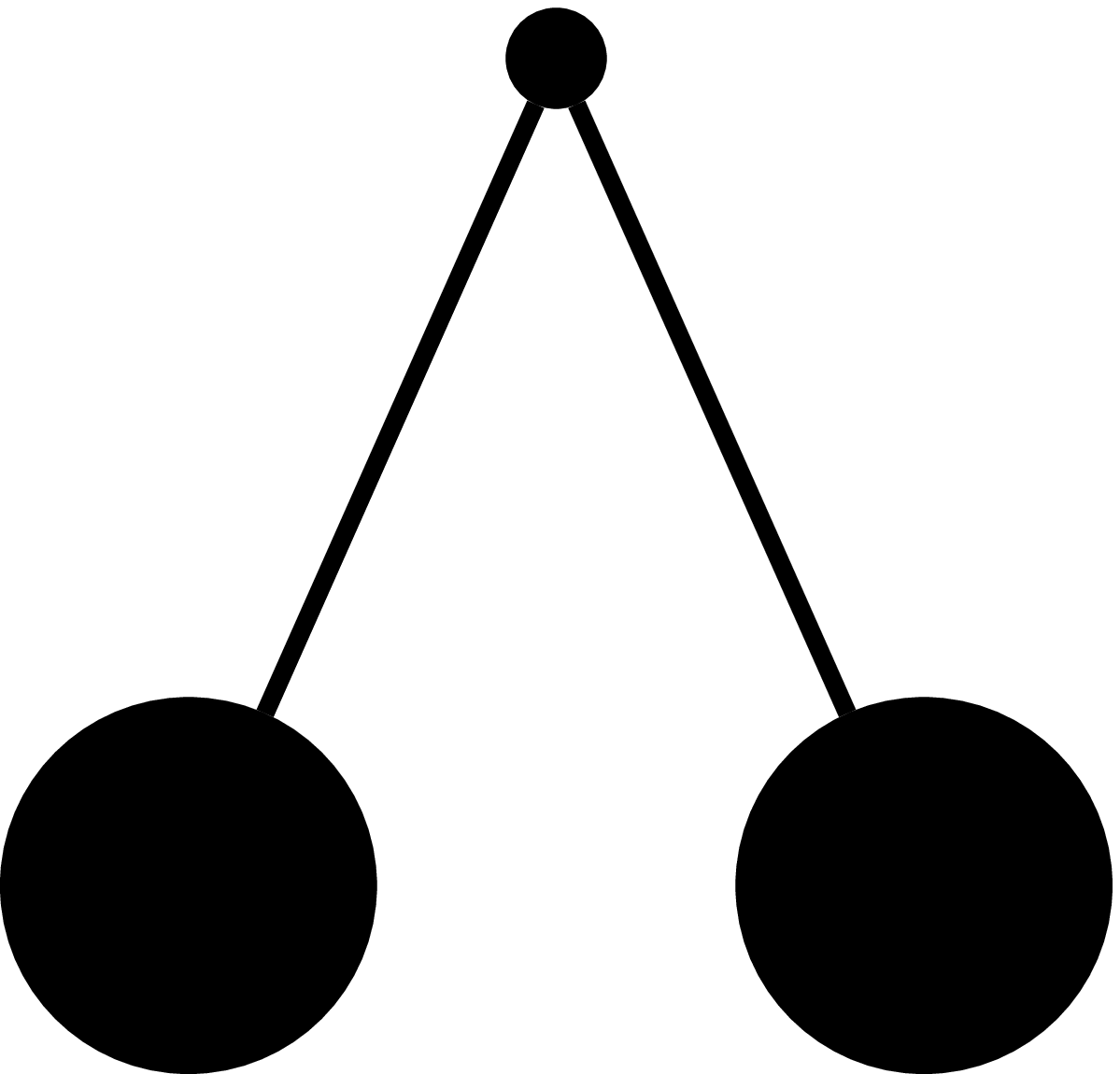} & &
    \includegraphics[scale=0.15]{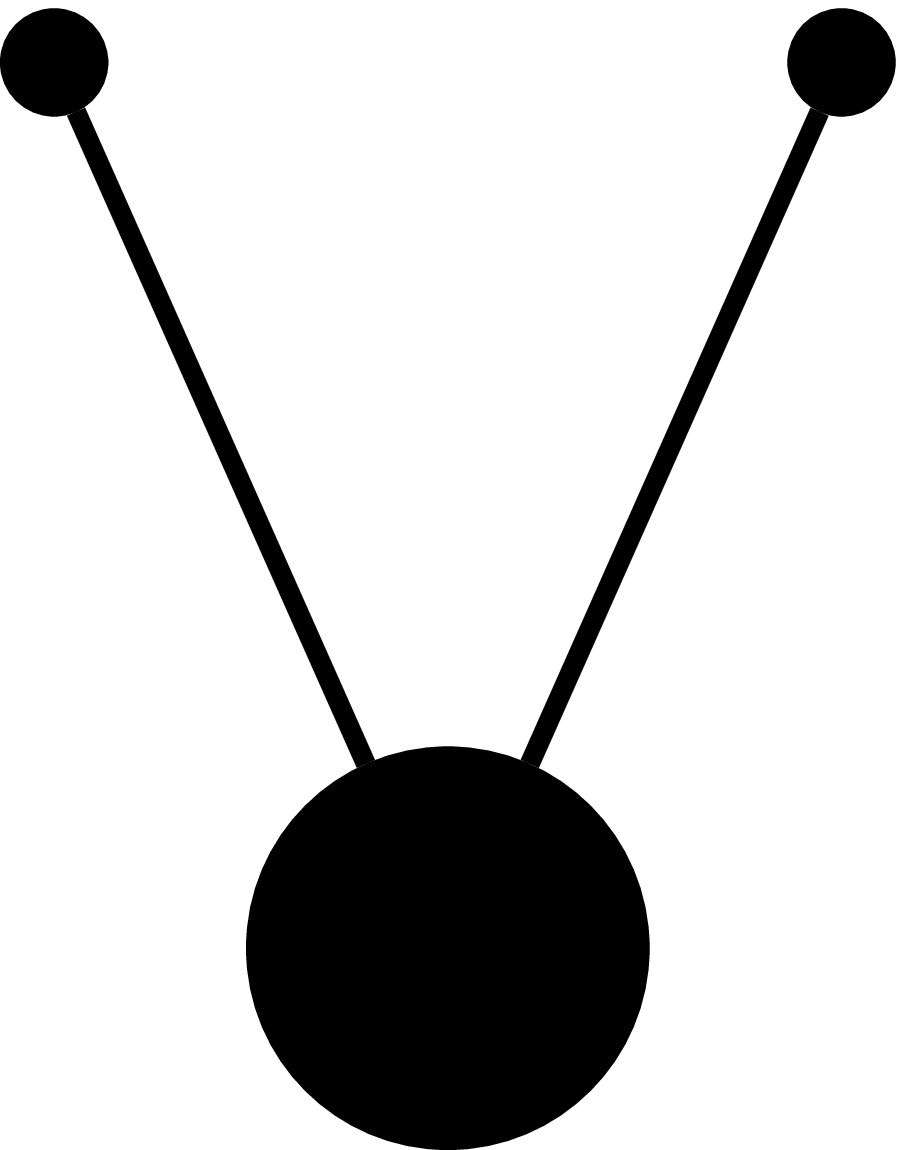} \\
$\Ho_{{\rm I}}$ & &
$\Ho_{{\rm II}}$ & & $\Ho_{{\rm III}}$ \\
\end{tabular}
\end{center}
\caption{
The Hoffman graphs $\Ho_{{\rm I}}$, 
$\Ho_{{\rm II}}$, and $\Ho_{{\rm III}}$
}
\label{fig:Hoffman123}
\end{figure}

\begin{lemma}[{\cite[Lemma 3.4]{hlg}}]\label{lm:BH}
The diagonal entry 
$B(\Ho)_{xx}$ of the matrix $B(\Ho)$ is equal to $-|N^f_{\Ho}(x)|$,
and the off-diagonal entry $B(\Ho)_{xy}$ is equal to 
$A_{xy} - |N^f_{\Ho}(x) \cap N^f_{\Ho}(y)|$. 
\end{lemma}

\begin{lemma}[{\cite[Corollary 3.3]{hlg}}]\label{lm:001}
If $\frak{G}$ is an induced Hoffman subgraph of a Hoffman graph $\frak{H}$, 
then $\lambda_{\min}(\frak{G})\geq\lambda_{\min}(\frak{H})$ holds. 
In particular, if $\Gamma$ is the slim
subgraph of $\Ho$, then $\lambda_{\min}(\Gamma)\geq
\lambda_{\min}(\Ho)$.
\end{lemma}

\subsection{Decompositions of Hoffman graphs}

\begin{defi}\label{df:002}
A {\em decomposition} of a Hoffman graph $\Ho$ 
is a family $\{\Ho^i\}_{i=1}^n$ of Hoffman subgraphs of $\Ho$
satisfying the following conditions:
\begin{itemize}
\item[{\rm (i)}] 
$V(\Ho)=\bigcup_{i=1}^n V(\Ho^i)$;
\item[{\rm (ii)}] 
$V^s(\Ho^i) \cap V^s(\Ho^j)=\emptyset$ if $i\neq j$;
\item[{\rm (iii)}] 
if $x\in V^s(\Ho^i)$, $y\in V^f(\Ho)$, and $\{x, y\} \in E(\Ho)$, 
then $y \in V(\Ho^i)$; 
\item[{\rm (iv)}] 
if $x \in V^s(\Ho^i)$, $y\in V^s(\Ho^j)$, and $i \neq j$, 
then $|N^f_{\Ho}(x) \cap N^f_{\Ho}(y)| \leq 1$, 
and $|N^f_{\Ho}(x) \cap N^f_{\Ho}(y)| = 1$ 
if and only if $\{x, y\} \in E(\Ho)$.
\end{itemize}
If a Hoffman graph $\Ho$ has a decomposition $\{\Ho^i\}_{i=1}^n$, 
then we write $\Ho=\biguplus_{i=1}^n \Ho^i$. 
\end{defi}

\begin{exa} 
The (slim) complete graph $K_n$ is precisely
the slim subgraph of the Hoffman graph 
$\Ho=\biguplus_{i=1}^n \Ho^i$ 
where each $\Ho^i$ is isomorphic to $\Ho_{{\rm I}}$, 
sharing the unique fat vertex. 

Ordinary line graphs are precisely 
the slim subgraphs of Hoffman graphs
$\Ho=\biguplus_{i=1}^n \Ho^i$,
where each $\Ho^i$ is isomorphic to $\Ho_{{\rm II}}$. 

The (slim) cocktail party graph 
$CP(n)=K_{n \times 2}$
is precisely 
the slim subgraph of the Hoffman graph 
$\Ho=\biguplus_{i=1}^n \Ho^i$ 
where each $\Ho^i$ is isomorphic to $\Ho_{{\rm III}}$, 
sharing the unique fat vertex. 

Generalized line graphs are precisely 
the slim subgraphs of Hoffman graphs
$\Ho=\biguplus_{i=1}^n \Ho^i$,
where each $\Ho^i$ is isomorphic to 
$\Ho_{{\rm II}}$ or $\Ho_{{\rm III}}$
(see \cite{hlg}). 
\end{exa}

\begin{defi}
A Hoffman graph $\Ho$ is said to be {\em decomposable} 
if 
$\Ho$ has a decomposition $\{\Ho^i\}_{i=1}^n$ with $n \geq 2$. 
We say $\Ho$ is {\em indecomposable} 
if $\Ho$ is not decomposable. 
\end{defi}

\begin{exa}
A disconnected Hoffman graph is decomposable.
\end{exa}

\begin{defi}
Let $\alpha$ be a negative real number. 
Let $\Ho$ be a Hoffman graph with $\lambda_{\min}(\Ho) \geq \alpha$.
The Hoffman graph $\Ho$ is said to be {\em $\alpha$-re\-duc\-i\-ble} 
if there exists a Hoffman graph $\Ho'$ containing $\Ho$ as an
induced Hoffman subgraph such that there is a decomposition
$\{\Ho^i\}_{i=1}^2$ of $\Ho'$ with $\lambda_{\min}(\Ho^i)\geq\alpha$
and 
$V^s(\Ho^i) \cap V^s(\Ho) \neq \emptyset$ $(i=1,2)$. 
We say $\Ho$ is {\em $\alpha$-ir\-re\-duc\-i\-ble} 
if $\lambda_{\min}(\Ho) \geq \alpha$ and
$\Ho$ is not $\alpha$-re\-duc\-i\-ble. 
A Hoffman graph $\Ho$ is said to be {\em reducible} 
if $\Ho$ is $\lambda_{\min}(\Ho)$-re\-duc\-i\-ble. 
We say $\Ho$ is {\em irreducible} 
if $\Ho$ is not reducible. 
\end{defi}

\begin{lemma}[{\cite[Lemma 2.12]{JKMT}}]\label{lm:002}
If a Hoffman graph $\Ho$ has a decomposition $\{\Ho^i\}_{i=1}^n$,
then 
$\lambda_{\min}(\Ho)=\min\{\lambda_{\min}(\Ho^i) \mid 1 \leq i \leq n\}$. 
In particular,
an irreducible Hoffman graph is indecomposable.
\end{lemma} 

\begin{exa}\label{rem:00} 
For a non-neg\-a\-tive integer $t$, 
let $\mathfrak{K}_{1,t}$ be 
the connected Hoffman graph having exactly one slim vertex 
and $t$ fat vertices, i.e., 
\[
\mathfrak{K}_{1,t} 
= (V^s(\mathfrak{K}_{1,t}), V^f(\mathfrak{K}_{1,t}), E(\mathfrak{K}_{1,t}))
= (\{v\}, \{f_1, \ldots, f_t\}, 
\{\{v,f_i\} \mid i=1, \ldots t\}). 
\]
Then 
$\mathfrak{K}_{1,t}$ is irreducible 
and $\lambda_{\min}(\mathfrak{K}_{1,t}) = -t$. 
\end{exa}

\begin{exa}\label{ex:002}
By Example \ref{rem:00}, 
the Hoffman graphs $\Ho_{{\rm I}}$ $(\cong \mathfrak{K}_{1,1})$
 and $\Ho_{{\rm II}}$ $(\cong \mathfrak{K}_{1,2})$
are irreducible.  
The Hoffman graph $\Ho_{{\rm III}}$ is also irreducible. 
\end{exa}

\begin{exa}\label{ex:H-4}
Let $\Ho_{{\rm IV}}$ be the Hoffman graph defined by 
$V^s(\Ho_{{\rm IV}}) = \{v_1,v_2\}$, 
$V^f(\Ho_{{\rm IV}}) = \{f_1,f_2\}$, 
and 
$E(\Ho_{{\rm IV}}) = \{\{v_1,v_2\}, \{v_1,f_1\}, \{v_2, f_2\}\}$. 
The Hoffman graph $\Ho_{{\rm IV}}$ is indecomposable 
but reducible. 
Indeed, it is clear that $\Ho_{{\rm IV}}$ is indecomposable.
Let $\Ho'$ be the Hoffman graph 
obtained from $\Ho_{{\rm IV}}$ 
by adding a new fat vertex $f_3$
and two edges $\{v_1, f_3 \}$ and $\{v_2, f_3\}$. 
The Hoffman graph $\Ho'$ is the sum of two copies of $\Ho_{{\rm II}}$, 
where the newly added fat vertex is shared by both copies, 
that is, $\Ho'$ is decomposable. 
Furthermore, 
$\lambda_{\min}(\Ho_{{\rm II}}) = \lambda_{\min}(\Ho_{{\rm IV}}) =-2$. 
Hence $\Ho_{{\rm IV}}$ is reducible. 
\end{exa}

\begin{prop}
Let $G$ be a slim graph with at least two vertices.
If $G$ has maximum degree $k$, then $G$ is $(-k)$-re\-ducible.
\end{prop}

\begin{proof}
Let $G$ be a slim graph with maximum degree $k$. 
We define a Hoffman graph $\Ho$ by adding 
a fat vertex for each edge $e$ of $G$ 
and joining it to the two end vertices of $e$. 
Note that $G$ is the slim subgraph of $\Ho$.
For each slim vertex $x\in V^s(\Ho)$, 
let $\Ho^{x}$ be the Hoffman subgraph 
of $\Ho$ induced by $\{x\} \cup N_{\Ho}^f(x)$. 
Then $\Ho^{x}$ is isomorphic to the Hoffman graph 
$\mathfrak{K}_{1,\deg_G(x)}$ defined in Example \ref{rem:00}, 
and we can check that 
$\Ho = \biguplus_{x \in V^s(\Ho)} \Ho^{x}$. 
Since the maximum degree of $G$ is $k$, 
$\lambda_{\min}(\Ho^{x})=-\deg_G(x)\geq-k$. 
Thus $G$ is $(-k)$-re\-ducible. 
\end{proof}

\begin{defi}
Let $\mathcal{H}$ be a family of isomorphism classes 
of Hoffman graphs. 
An {\em $\mathcal{H}$-line graph} is 
an induced Hoffman subgraph of a Hoffman graph 
which has a decomposition $\{\Ho^i\}_{i=1}^n$
such that the isomorphism class of $\Ho^i$ belongs to
$\mathcal{H}$ for all $i=1,\dots,n$. 
\end{defi}

\subsection{The special graphs of Hoffman graphs}\label{sec:004}

\begin{defi}
An {\it edge-signed graph} $\s$ is a pair $(S,\textrm{sgn})$ 
of a graph $S$ and 
a map $\textrm{sgn}:E(S) \to \{+,-\}$. 
Let $V(\s):=V(S)$, 
$E^{+}(\s):= \textrm{sgn}^{-1}(+)$, 
and $E^{-}(\s):= \textrm{sgn}^{-1}(-)$. 
Each element in $E^{+}(\s)$ (resp.\ $E^{-}(\s)$) 
	is called a \emph{$(+)$-edge} 
	(resp.\ a \emph{$(-)$-edge}) of $\s$. 
We represent an edge-signed graph $\s$ also by the triple 
$(V(\s), E^{+}(\s), E^{-}(\s))$. 

An edge-signed graph $\s'=(S',\textrm{sgn}')$ is called 
an {\em induced edge-signed subgraph} 
of an edge-signed graph $\s=(S,\textrm{sgn})$ 
if $S'$ is an induced subgraph 
of $S$ and 
$\textrm{sgn}(e)=\textrm{sgn}'(e)$ 
holds for any edge $e$ of $S'$. 

Two edge-signed graphs $\s$ and $\s'$ 
are said to be \emph{isomorphic} 
if there exists a bijection $\phi: V(\s) \to V(\s')$ 
such that $\{u,v\} \in E^+(\s)$ if and only if 
$\{\phi(u), \phi(v) \} \in E^+(\s')$ and 
that $\{u,v\} \in E^-(\s)$ if and only if 
$\{\phi(u), \phi(v) \} \in E^-(\s')$.

An edge-sined graph $\s$ 
is said to be \emph{connected} (resp.\ \emph{disconnected})
if the graph $(V(\s),E^+(\s) \cup E^-(\s))$ is connected 
(resp.\ disconnected). 
\end{defi}

\begin{exa}
A connected edge-signed graph with at most two vertices 
is isomorphic to one of the edge-signed graphs 
$\s_{1,1}$, $\s_{2,1}$, and $\s_{2,2}$, 
where 
\[
\begin{array}{lll}
V(\s_{1,1}) = \{v_1\}, &
E^+(\s_{1,1}) = \emptyset, &
E^-(\s_{1,1}) = \emptyset, \\
V(\s_{2,1}) = \{v_1,v_2\}, &
E^+(\s_{2,1}) = \{\{v_1,v_2\}\}, &
E^-(\s_{2,1}) = \emptyset, \\
V(\s_{2,2}) = \{v_1,v_2\}, &
E^+(\s_{2,2}) = \emptyset, &
E^-(\s_{2,2}) = \{\{v_1,v_2\}\}. 
\end{array}
\]
(see Figure~\ref{fig:002}  in which 
we draw an edge-signed graph by depicting $(+)$-edges as full lines
and $(-)$-edges as dashed lines).
\end{exa}

\begin{figure}[!h]
\begin{center}
\begin{tabular}{p{1.5cm}p{1.5cm}p{1.5cm}}
\includegraphics[scale=0.6]{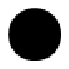} &
\includegraphics[scale=0.6]{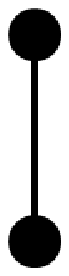} &
\includegraphics[scale=0.6]{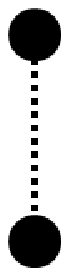} \\
$\s_{1,1}$ & $\s_{2,1}$ & $\s_{2,2}$ \\
\end{tabular}
\end{center}
\caption{The connected edge-signed graphs 
with at most two vertices}
\label{fig:002}
\end{figure}

\begin{defi}
The {\it special graph} of a Hoffman graph $\Ho$ is 
the edge-signed graph 
\[
\s(\Ho) := (V(\s(\Ho)), E^{+}(\s(\Ho)), E^{-}(\s(\Ho))) 
\]
where $V(\s(\Ho)):=V^s(\Ho)$ and 
\[
\begin{array}{l}
E^+(\s(\Ho)) := \{ \{u,v\} \mid 
u,v \in V^s(\Ho), 
u \neq v, 
\{u,v\} \in E(\Ho), 
N^f_{\Ho}(u) \cap N^f_{\Ho}(v) = \emptyset \}, \\
E^-(\s(\Ho)) := \{ \{u,v\} \mid 
u,v\in V^s(\Ho), 
u \neq v, 
\{u,v\} \notin E(\Ho), 
N^f_{\Ho}(u) \cap N^f_{\Ho}(v) \neq \emptyset \}. 
\end{array}
\]
\end{defi}

\begin{lemma}[{\cite[Lemma 3.4]{JKMT}}]\label{lm:007} 
A Hoffman graph $\Ho$ is indecomposable 
if and only if its special graph $\s(\Ho)$ is connected. 
\end{lemma}

\begin{defi}
For an edge-signed graph $\s$, 
we define its \emph{signed adjacency matrix} $M(\s)$ 
by 
\[
(M(\s))_{uv}=
\begin{cases}
1  & \text{if } \{u,v\} \in E^+(\s), \\
-1 & \text{if } \{u,v\} \in E^-(\s), \\
0  & \text{otherwise.} 
\end{cases} 
\]
We denote by $\lambda_{\min}(\s)$ the smallest eigenvalue of $M(\s)$. 
\end{defi}

We remark that P.~J.~Cameron, J.~J.~Seidel, and S.~V.~Tsaranov studied 
the eigenvalues of edge-signed graphs in \cite{CST}. 

\begin{lemma}\label{lm:induced-esg}
If $\s'$ is an induced edge-signed subgraph 
of an edge-signed graph $\s$, 
then $\lambda_{\min}(\s') \geq \lambda_{\min}(\s)$. 
\end{lemma}

\begin{proof}
Since $M(\s')$ is a principal submatrix of $M(\s)$, 
the lemma holds. 
\end{proof}

\begin{lemma}\label{p:002}
Let $\Ho$ be a Hoffman graph in which 
any two distinct slim vertices have at most one common fat neighbor. 
Then 
\[
M(\s(\Ho))=B(\Ho) + D(\Ho), 
\]
where $D(\Ho)$ is the diagonal matrix defined by 
$D(\Ho)_{xx}:=|N^f_{\Ho}(x)|$ for $x \in V^s(\Ho)$.
\end{lemma}
\begin{proof}
This follows immediately from the definitions and Lemma~\ref{lm:BH}. 
\end{proof}

\begin{lemma}\label{cor:MS-B}
If $\Ho$ is a fat Hoffman graph 
with smallest eigenvalue greater than $-3$, 
then 
$\lambda_{\min} (\s(\Ho)) \geq \lambda_{\min} (\Ho) +1$. 
\end{lemma}
\begin{proof}
If some two distinct slim vertices of $\Ho$ have two common fat
neighbors, then $\Ho$ contains an induced subgraph with smallest
eigenvalue at most $-3$. This contradicts the assumption by
Lemma~\ref{lm:001}. Thus the hypothesis of Lemma~\ref{p:002}
is satisfied. Since $\Ho$ is fat, the smallest eigenvalue
of $M(\s(\Ho))=B(\Ho)+D(\Ho)$ is at least $\lambda_{\min}(\Ho)+1$
by \cite[Corollary 4.3.3]{HJ}, proving the desired inequality.
\end{proof}

\section{Main Results}
\subsection{The edge-signed graphs with smallest eigenvalue at least $-\tau$}

\begin{defi}\label{df:Qpqr}
Let $p$, $q$, and $r$ be non-neg\-a\-tive integers with $ p+q\le r $.
Let $V_p$, $V_q$, and $V_r$ be mutually disjoint sets such that $|V_i|=i$ 
where $i \in \{p,q,r\}$. 
Let $U_p $ and $U_q $ be subsets of $V_r$ 
satisfying $|U_p|=p$, $|U_q|=q$, and $U_p \cap U_q=\emptyset$.
Let $\mathcal{Q}_{p,q,r}$ be the edge-signed graph defined by 
\begin{eqnarray*}
V(\mathcal{Q}_{p,q,r}) &:=& V_p \cup V_q \cup V_r, \\
E^+(\mathcal{Q}_{p,q,r}) &:=& \{\{u,v\} \mid u \in U_p, v \in V_p \}
\cup 
\{\{v,v'\} \mid v,v' \in V_r, v \neq v' \}, \\
E^-(\mathcal{Q}_{p,q,r}) &:=& \{\{u,v\} \mid u \in U_q, v \in V_q \} 
\end{eqnarray*}
(see Figure \ref{fig:001} for an illustration). 
\end{defi}

\begin{figure}[!h]
\begin{center}
\includegraphics[scale=0.2]{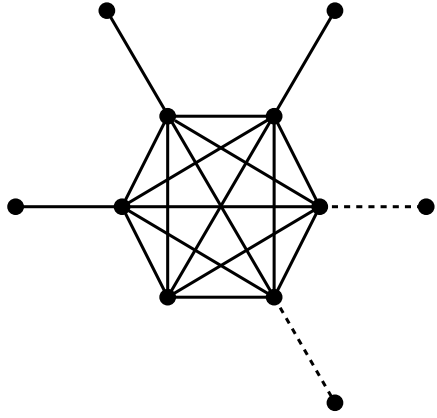}
\end{center}
\caption{$\mathcal{Q}_{3,2,6}$}
\label{fig:001}
\end{figure}

\begin{lemma}\label{lm:MQpqr}
For any non-neg\-a\-tive integers $p$, $q$, and $r$ with $p+q \leq r$, 
$\lambda_{\min}(\mathcal{Q}_{p,q,r}) 
\geq -\tau$.
\end{lemma}

\begin{proof}
Let
\[
M(\mathcal{Q}_{r,r,2r})=
\begin{bmatrix}
0&0&I&0\\
0&0&0&-I\\
I&0&J-I&J\\
0&-I&J&J-I
\end{bmatrix}.
\]
Multiplying 
\[
\begin{bmatrix}
I&0&xI&0\\
0&I&0&-xI\\
0&0&I&0\\
0&0&0&I
\end{bmatrix}
\]
from the left to $xI-M(\mathcal{Q}_{r,r,2r})$, we find
\begin{align*}
|xI-M(\mathcal{Q}_{r,r,2r})|&=
(-1)^r
\begin{vmatrix}
(x^2+x-1)I&-(x^2+x-1)I\\
xJ&-(x^2+x-1)I+xJ
\end{vmatrix}
\\&=
(x^2+x-1)^r
\begin{vmatrix}
I&I\\
xJ&(x^2+x-1)I-xJ
\end{vmatrix}
\\&=
(x^2+x-1)^r
|(x^2+x-1)I-2xJ|
\\&=
(x^2+x-1)^{2r-1}
(x^2-(2r-1)x-1).
\end{align*}
In particular, we obtain 
$\lambda_{\min}(\mathcal{Q}_{r,r,2r}) = -\tau$.
Since $p \leq r$ and $q \leq r$, 
$\mathcal{Q}_{r,r,2r}$ has an induced 
edge-signed subgraph 
isomorphic to $\mathcal{Q}_{p,q,r}$. 
By Lemma \ref{lm:induced-esg}, 
$\lambda_{\min}(\mathcal{Q}_{p,q,r}) 
\geq \lambda_{\min}(\mathcal{Q}_{r,r,2r})
=-\tau$.
\end{proof}

\begin{exa} \label{ex:MS}
Let $\mathcal{T}_{1}$ and $\mathcal{T}_{2}$ 
be the edge-signed triangles having exactly one $(+)$-edge 
and exactly two $(+)$-edges, respectively, i.e., 
$V(\mathcal{T}_{1}) = V(\mathcal{T}_{2}) = \{v_1,v_2,v_3\}$, 
$E^+(\mathcal{T}_{1}) = E^-(\mathcal{T}_{2}) = \{\{v_1,v_2\}\}$, and 
$E^-(\mathcal{T}_{1}) = E^+(\mathcal{T}_{2}) 
= \{\{v_1,v_3\},$ $\{v_2,v_3\} \}$ 
(see Figure \ref{fig:S30}).

For $\epsilon_1, \epsilon_2, \epsilon_3 \in\{1,-1\}$ 
and $\delta\in\{0,\pm1\}$, 
let 
$\s_1(\epsilon_1,\epsilon_2,\epsilon_3)$,
$\s_2(\epsilon_1,\epsilon_2,\delta)$,
$\s_3(\epsilon_1,\epsilon_2,\delta)$, and 
$\s_4(\epsilon_1,\epsilon_2)$ 
be the edge-signed graphs in Figure \ref{fig:S1-5}, 
where an edge with label $1$ (resp.\ $-1$) 
represents a $(+)$-edge (resp.\ a $(-)$-edge) 
and an edge with label $0$ represents 
a non-adjacent pair. 

It can be checked that 
each of the edge-signed graphs 
$\mathcal{T}_2$, 
$\s_1(\epsilon_1,\epsilon_2,\epsilon_3)$,
$\s_2(\epsilon_1,\epsilon_2,\delta)$,
$\s_3(\epsilon_1,\epsilon_2,\delta)$, 
$\s_4(\epsilon_1,\epsilon_2)$ 
has the smallest eigenvalue less than $-\tau$. 
\end{exa}

\begin{figure}[!h]
\begin{center}
\begin{tabular}{ccc}
\includegraphics[scale=0.8]{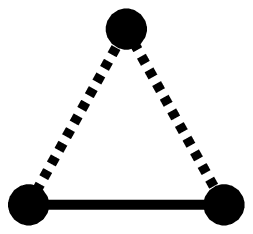} & &
\includegraphics[scale=0.8]{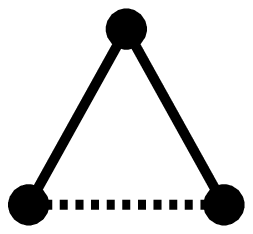} \\
$\mathcal{T}_1$ & & $\mathcal{T}_2$ \\
\end{tabular}
\end{center}
\caption{The edge-signed triangles $\mathcal{T}_{1}$
and $\mathcal{T}_{2}$}
\label{fig:S30}
\end{figure}

\begin{figure}[!h]
\psfrag{e1}{$\epsilon_1$}
\psfrag{e2}{$\epsilon_2$}
\psfrag{e3}{$\epsilon_3$}
\psfrag{d}{$\delta$}
\begin{center}
\begin{tabular}{cc}
\includegraphics[scale=0.8]{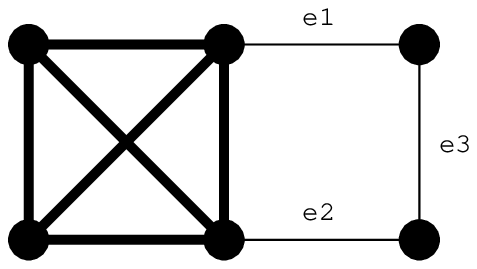} &
\includegraphics[scale=0.8]{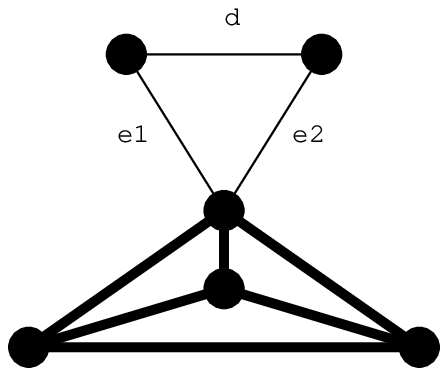} \\
$\s_1(\epsilon_1,\epsilon_2,\epsilon_3)$ & 
$\s_2(\epsilon_1,\epsilon_2,\delta)$ \\
& \\
& \\
\includegraphics[scale=0.8]{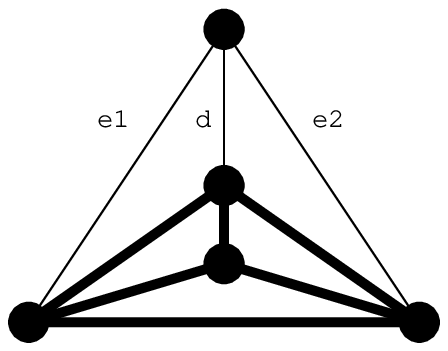} &
\includegraphics[scale=0.8]{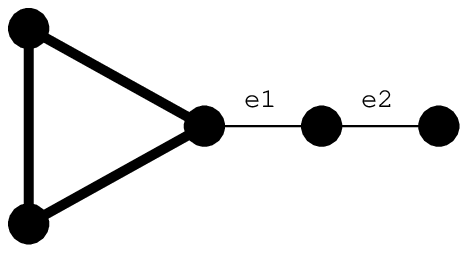} \\
$\s_3(\epsilon_1,\epsilon_2,\delta)$ &
$\s_4(\epsilon_1,\epsilon_2)$ \\
\end{tabular}
\end{center}
\caption{Edge-signed graphs with smallest eigenvalue less than $-\tau$}
\label{fig:S1-5}
\end{figure}

\begin{figure}[!h]
\begin{center}
\begin{tabular}{ccccc}
    \includegraphics[scale=0.6]{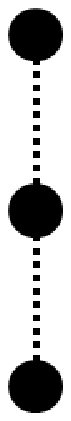}&
    \includegraphics[scale=0.6]{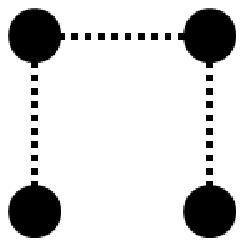}&
    \includegraphics[scale=0.6]{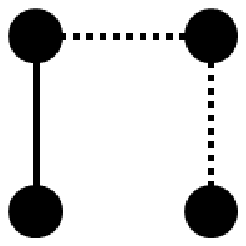}&
    \includegraphics[scale=0.6]{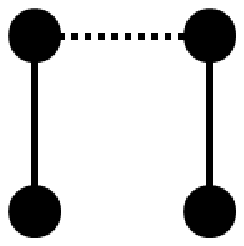}&
    \includegraphics[scale=0.6]{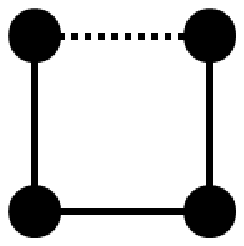}\\
    $\s_{3,1}$&$\s_{4,1}$&$\s_{4,2}$&$\s_{4,3}$&$\s_{4,4}$\\
    &&&&\\
    &&&&\\
    \includegraphics[scale=0.6]{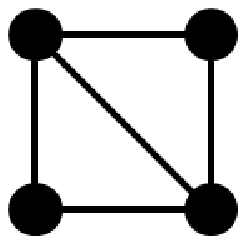}&
    \includegraphics[scale=0.6]{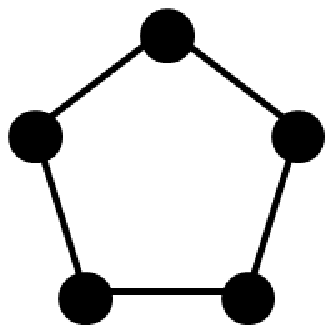}&
    \includegraphics[scale=0.6]{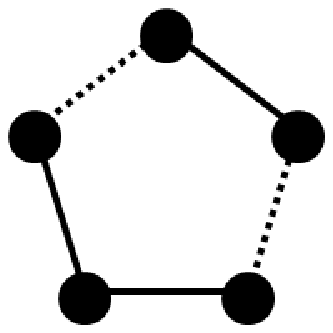}&
    \includegraphics[scale=0.6]{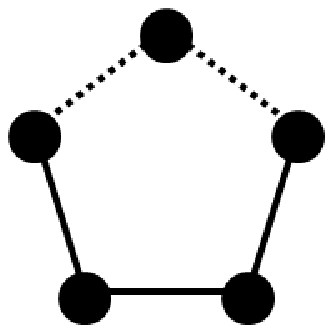}&
    \includegraphics[scale=0.6]{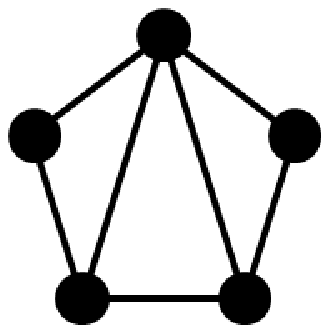}\\
    $\s_{4,5}$&$\s_{5,1}$&$\s_{5,2}$&$\s_{5,3}$&$\s_{5,4}$\\
    &&&&\\
    &&&&\\
    \includegraphics[scale=0.6]{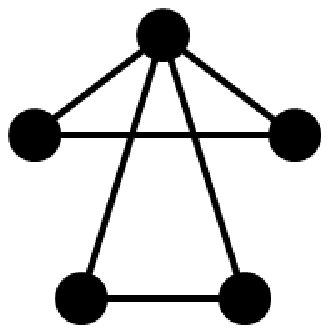}&
    \includegraphics[scale=0.6]{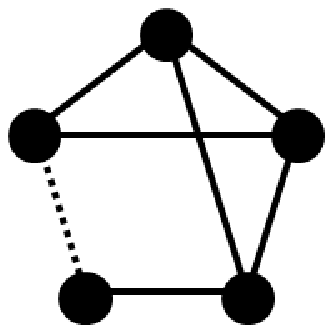}&
    \includegraphics[scale=0.6]{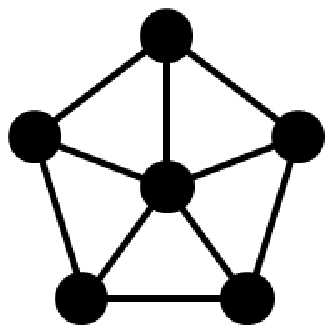}&
    \includegraphics[scale=0.6]{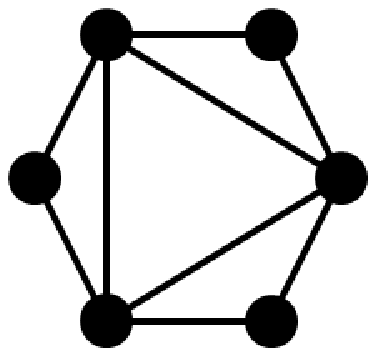}&
    \includegraphics[scale=0.6]{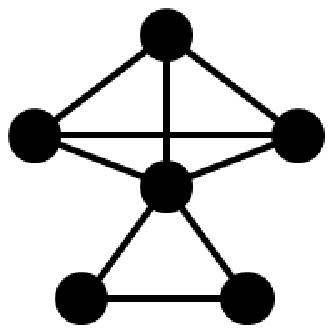}\\
    $\s_{5,5}$&$\s_{5,6}$&$\s_{6,1}$&$\s_{6,2}$&$\s_{6,3}$\\
\end{tabular}
\end{center}
\caption{The $15$ connected edge-signed graphs
with smallest eigenvalue at least 
$-\tau$ not containing $\mathcal{T}_1$, 
other than $\mathcal{Q}_{p,q,r}$
where $p,q,r$ are some non-neg\-a\-tive integers with 
$p+q\leq r$ 
}
\label{fig:004}
\end{figure}

\begin{theorem}\label{p:005++}
Let $\s$ be a connected edge-signed graph with
$\lambda_{\min}(\s)\geq 
-\tau$.
Assume that $\s$ does not contain 
an induced edge-signed subgraph isomorphic to $\mathcal{T}_1$. 
Then 
either $\s$ is isomorphic to $\mathcal{Q}_{p,q,r}$ 
for some non-neg\-a\-tive integers $p,q,r$ with 
$p+q\leq r$, or
$\s$ has at most $6$ vertices
and is isomorphic to one of the $15$ edge-signed graphs 
in Figure \ref{fig:004}. 
\end{theorem}

\begin{proof}
By using computer \cite{MAGMA}, we 
checked 
that 
the theorem holds when $|V(\s)| \leq 7$. 
We prove the assertion 
by induction on $|V(\s)|$. 
Assume that the assertion 
holds for $|V(\s)|=n$ $(\geq 7)$. 
Suppose that $|V(\s)|=n+1$. 
It follows from Problem 6(a) in Section~6 of \cite{exercises}
that there exists a vertex $v$ 
which is not a cut vertex of $\s$.
Then $\s-v$ is connected, where $\s-v$ is the edge-signed subgraph
induced by $V(\s)\setminus\{v\}$.
Since
$\lambda_{\min}(\s-v) \geq \lambda_{\min}(\s) 
\geq 
-\tau$,
the inductive hypothesis implies that
$\s-v$ is isomorphic to $\mathcal{Q}_{p,q,r}$ 
for some $p,q,r$ with $p+q+r=n$. 
Thus $\s$ is the edge-signed graph 
obtained from $\mathcal{Q}_{p,q,r}$ 
by adding the vertex $v$ and signed edges 
between $v$ and some vertices in $\mathcal{Q}_{p,q,r}$. 
Note that $r \geq 4$ since $n=p+q+r \geq 7$ 
and $p+q \leq r$. 

We claim that either $v$ is adjacent to 
only one vertex of $V_r$, or to all the vertices of $V_r$.
Note that $\s$ cannot contain any of 
the edge-signed graphs 
$\mathcal{T}_2$, 
$\s_1(\epsilon_1,\epsilon_2,\epsilon_3)$,
$\s_2(\epsilon_1,\epsilon_2,\delta)$,
$\s_3(\epsilon_1,\epsilon_2,\delta)$, 
$\s_4(\epsilon_1,\epsilon_2)$ 
in Example~\ref{ex:MS}.
If $v$ is adjacent to none of the vertices of $V_r$, then
$\s$ contains $\s_4(\epsilon_1,\epsilon_2)$ 
as an induced edge-signed subgraph,
a contradiction.
If the number of neighbors of $v$ in $V_r$ is at least $2$ and
less than $r$, then
$\s$ contains $\s_3(\epsilon_1,\epsilon_2,\delta)$ 
as an induced edge-signed subgraph,
a contradiction. 
Thus the claim holds.

Now, if
$v$ is adjacent to only one vertex of $V_r$, then 
the unique neighbor of $v$ in $V_r$ is in $V_r \setminus (U_p\cup U_q)$.
Indeed, otherwise we would 
find $\s_2(\epsilon_1,\epsilon_2,\delta)$ 
as an induced edge-signed subgraph,
a contradiction.
Also, 
$v$ is adjacent to none of the vertices of $V_p\cup V_q$
since otherwise we would 
find $\s_1(\epsilon_1,\epsilon_2,\epsilon_3)$ 
as an induced edge-signed subgraph, 
a contradiction.
Thus $\s$ is isomorphic to $\mathcal{Q}_{p+1,q,r}$ or
$\mathcal{Q}_{p,q+1,r}$.

Suppose that $v$ is adjacent to all the vertices of $V_r$.
Since $V_r$ is a clique consisting of $(+)$-edges only,
the assumption implies that $v$ is incident with at most one
$(-)$-edge to $V_r$. If there is a vertex of $V_r$ joined to
$v$ by a $(-)$-edge, then we find
$\mathcal{T}_2$ 
as an induced edge-signed subgraph, 
a contradiction.
Thus all the edges from $v$ to $V_r$ are $(+)$-edges.
Now $v$ is adjacent to none of the vertices of $V_p\cup V_q$
since otherwise we would 
find $\s_3(\epsilon_1,\epsilon_2,0)$ 
as an induced edge-signed subgraph, 
a contradiction.
Thus $\s$ is isomorphic to $\mathcal{Q}_{p,q,r+1}$. 
Hence the theorem holds. 
\end{proof}

\begin{lemma}
The smallest eigenvalues of the signed adjacency matrices of 
the edge-signed graphs in Figure \ref{fig:004} 
are given as follows: 
\[
\lambda_{\min}(\s) = 
\left\{
\begin{array}{cll}
-\sqrt{2} 
& \approx -1.414213
& \quad \text{if } \s \in \{\s_{3,1}, \s_{4,4} \}, \\
\frac{1 - \sqrt{17}}{2} 
& \approx -1.561553
& \quad \text{if } \s \in \{\s_{4,5}, \s_{5,5}, \s_{5,6} \}, \\
1 + t 
& \approx -1.601679
& \quad \text{if } \s=\s_{6,3}, \\
-\tau 
& \approx -1.618034 
& \quad \text{otherwise}, \\
\end{array}
\right.
\]
where 
$t$ is the smallest zero of the polynomial 
$x^3 - 6x + 2$. 
\end{lemma}

\begin{proof}
This can be checked by a direct calculation. 
\end{proof}

\begin{remark}
Among edge-signed graphs in Figure \ref{fig:004}, 
the maximal ones with respect to taking 
induced edge-signed graphs 
are $\s_{4,1}$, $\s_{5,2}$, $\s_{5,3}$, $\s_{5,6}$, 
$\s_{6,1}$, $\s_{6,2}$, $\s_{6,3}$. 
\end{remark}

\subsection{The special graphs of fat $(-1-\tau)$-irreducible Hoffman graphs}

\begin{lemma}\label{lm:005}
Let $\Ho$ be a Hoffman graph with smallest eigenvalue at least 
$-1-\tau$. 
Then every slim vertex of $\Ho$ has at most two fat neighbors. 
\end{lemma}

\begin{proof}
If a slim vertex $v$ of $\Ho$ has at least $3$ fat neighbors,
then $\Ho$ contains an induced Hoffman subgraph 
isomorphic to the Hoffman graph $\mathfrak{K}_{1,3}$ 
(see Example \ref{rem:00}). 
By Lemma \ref{lm:001}, we have 
$\lambda_{\min}(\Ho) \leq \lambda_{\min}(\mathfrak{K}_{1,3}) = -3$,
which is a contradiction to 
$\lambda_{\min}(\Ho) \geq -1-\tau$.
\end{proof}

\begin{lemma}\label{lm:-D}
Let $\s$ be a connected edge-signed graph with three vertices. 
Let $D$ be a $3\times3$ diagonal matrix with diagonal entries 
$1$ or $2$ such that at least one of the diagonal entries is $2$. 
Then 
$M(\s) - D$ has the smallest eigenvalue less than
$-1-\tau$.
\end{lemma}

\begin{proof}
This can be checked by a direct calculation. 
\end{proof}

\begin{lemma}\label{lm:004}
Let $\Ho$ be a fat indecomposable 
Hoffman graph with smallest eigenvalue at least 
$-1-\tau$. 
If some slim vertex of $\Ho$ has at least two fat neighbors, 
then 
the special graph $\s(\Ho)$ 
of $\Ho$ is isomorphic to 
$\mathcal{Q}_{0,0,1}$, $\mathcal{Q}_{1,0,1}$, or $\mathcal{Q}_{0,1,1}$. 
\end{lemma}

\begin{proof}
In view of Lemma \ref{lm:001}, 
it suffices to show that every fat indecomposable Hoffman graph 
with three slim vertices, in which some slim vertex has two fat neighbors, 
has the smallest eigenvalue less than 
$-1-\tau$. 

Let $\Ho$ be such a Hoffman graph. Then $\s(\Ho)$
is connected by Lemma~\ref{lm:007} and
$B(\Ho)=M(\s(\Ho))-D$ for some diagonal matrix $D$
with diagonal entries $1$ or $2$ such that at least
one of the diagonal entries is $2$. Then 
we have a contradiction by Lemma~\ref{lm:-D}.
\end{proof}

\begin{exa}\label{ex:H-16-17}
Let $\Ho_{{\rm XVI}}$ 
and $\Ho_{{\rm XVII}}$ 
be the Hoffman graphs in Figure \ref{fig:FHG-1-2}. 
The special graphs of $\Ho_{{\rm XVI}}$ and $\Ho_{{\rm XVII}}$ 
are 
$\mathcal{Q}_{1,0,1}$ and $\mathcal{Q}_{0,1,1}$, respectively, 
and 
$\lambda_{\min}(\Ho_{{\rm XVI}}) 
= \lambda_{\min}(\Ho_{{\rm XVII}}) = -1-\tau$. 
\end{exa}

\begin{figure}[!h]
\begin{center}
\begin{tabular}{cc}
    \includegraphics[scale=0.15]{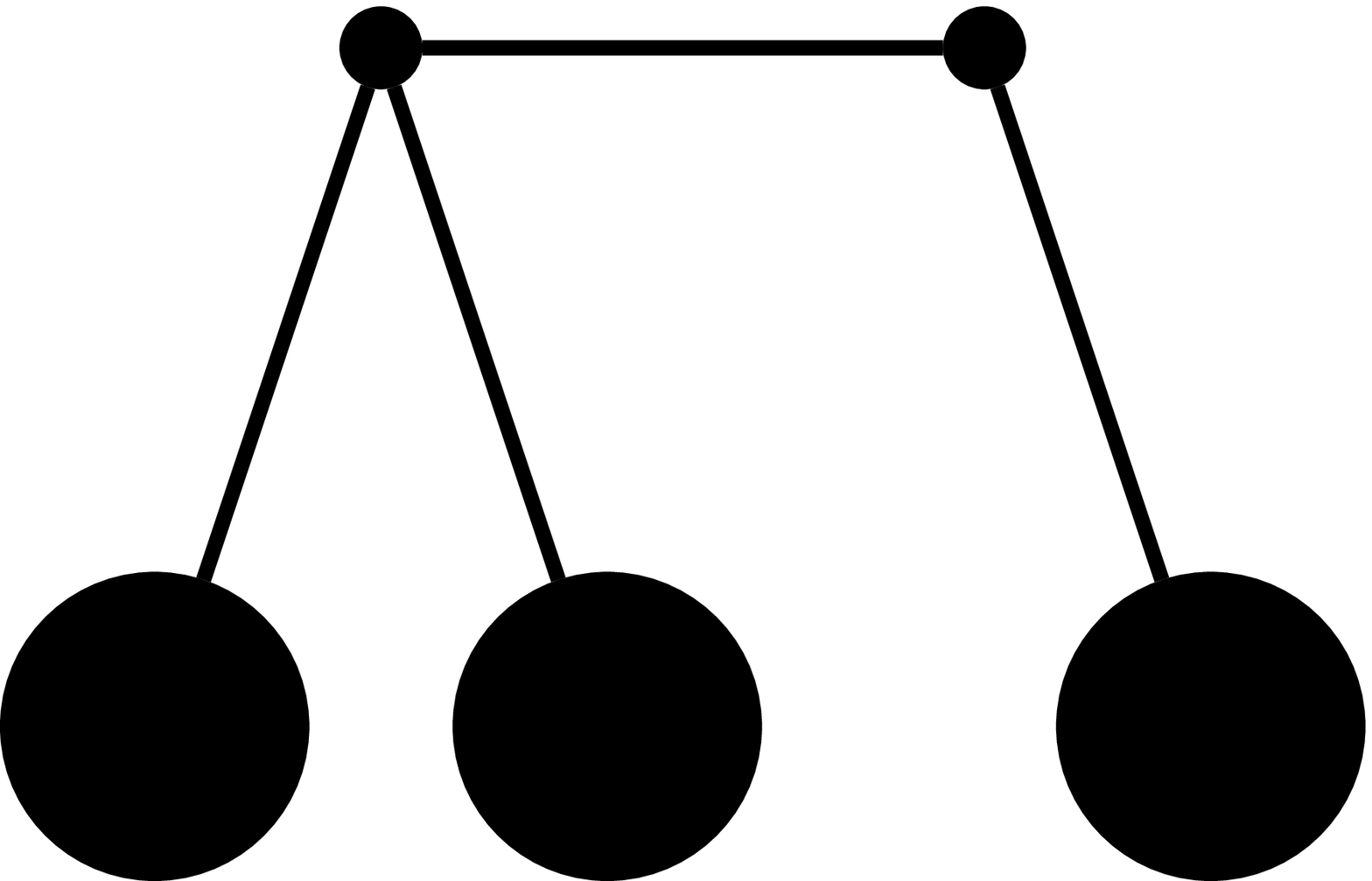} &
    \includegraphics[scale=0.15]{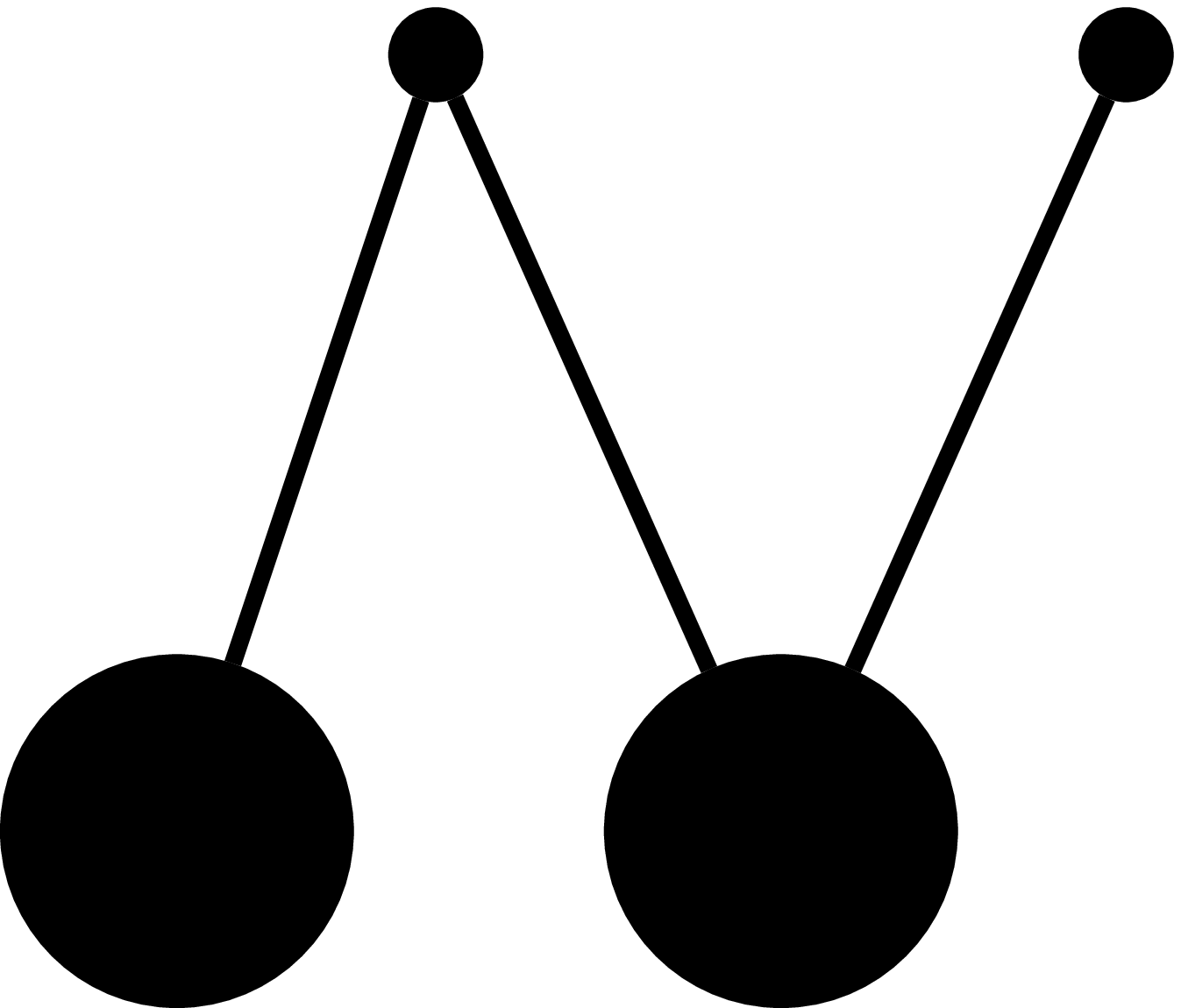} \\
$\Ho_{{\rm XVI}}$ & $\Ho_{{\rm XVII}}$ \\
\end{tabular}
\end{center}
\caption{
Fat indecomposable Hoffman graphs
}
\label{fig:FHG-1-2}
\end{figure}

\begin{lemma}\label{lm:012}
Let $\Ho$ be a Hoffman graph 
in which every slim vertex has at most one fat neighbor. 
Then the special graph $\s(\Ho)$ of $\Ho$ does not contain 
an induced edge-signed subgraph isomorphic to $\mathcal{T}_1$. 
\end{lemma}

\begin{proof}
Suppose that the special graph $\s(\Ho)$ of $\Ho$ contains 
$\mathcal{T}_1 = ( \{v_1,v_2,v_3\},$ 
$\{\{v_1,v_2\}\},$ $\{\{v_1,v_3\}, \{v_2,v_3\}\})$ 
as an induced edge-signed subgraph. 
Since $v_3$ is incident to a $(-)$-edge, 
$v_3$ must have a fat neighbor. 
Since every slim vertex of $\Ho$ has at most one fat neighbor, 
$v_3$ has a unique fat neighbor $f$. 
Then $f$ is adjacent to $v_1$ and $v_2$. 
This is a contradiction to $\{v_1,v_2\} \in E^+(\mathcal{T}_1)$. 
\end{proof}

\begin{lemma}\label{p:4-001}
Let $\Ho$ be a Hoffman graph 
in which every slim vertex has exactly one fat neighbor. 
If 
the special graph 
$\s(\Ho)$ of $\Ho$ is isomorphic to $\mathcal{Q}_{p,q,r}$ 
for some non-neg\-a\-tive integers $p,q,r$,
then $\Ho$ is an induced Hoffman subgraph 
of a Hoffman graph $\Ho'$ 
with $V^s(\Ho')=V^s(\Ho)$ 
which 
has a decomposition $\{\Ho^i\}_{i=1}^r$ such that
$\Ho^i$ is isomorphic to $\Ho_{{\rm XVI}}$,
$\Ho_{{\rm XVII}}$, or
$\Ho_{{\rm II}}$ for all $i=1,\dots,r$.
In particular, if $r \geq 2$, 
then $\Ho$ is $(-1-\tau)$-re\-duc\-i\-ble. 
\end{lemma}

\begin{proof}
By the assumption, 
$V^s(\Ho)=V(\s(\Ho))$ is partitioned into 
$V_p \cup V_q \cup V_r$ 
as Definition \ref{df:Qpqr}. 
Consider 
the Hoffman graph ${\Ho}'$ defined by 
$V^s({\Ho}') :=V^s(\Ho)$, 
$V^f({\Ho}') :=V^f(\Ho) \cup \{{f}^*\}$, 
and 
$E({\Ho}') :=E(\Ho) \cup \{\{v,{f}^*\} \mid v \in V_r\}$, 
where $f^*$ is a new fat vertex. 
Note that $\Ho$ is an induced Hoffman subgraph of $\Ho'$ 
with $V^s(\Ho)=V^s(\Ho')$. 
Then $\Ho'$ has a decomposition $\{\Ho^i\}_{i=1}^r$ with
$\Ho^i\cong \Ho_{{\rm XVI}}$ for $1\leq i\leq p$, 
$\Ho^i\cong \Ho_{{\rm XVII}}$ for $p < i\leq p+q$, and 
$\Ho^i\cong \Ho_{{\rm II}}$ for $p+q < i\leq p+q+r$
(see Examples \ref{ex:H-1-2-3} and \ref{ex:H-16-17}).
Since $r\geq2$ and 
each of the Hoffman graphs $\Ho^i$ has the smallest
eigenvalue at least $-1-\tau$, it follows that
$\Ho$ is $(-1-\tau)$-re\-duc\-i\-ble. 
\end{proof}

\begin{theorem}\label{thm:001-1}
Let $\Ho$ be a fat indecomposable 
Hoffman graph with smallest eigenvalue 
at least $-1-\tau$. 
Then the following hold:  
\begin{itemize}
\item[{\rm (i)}]
If some slim vertex of $\Ho$ has at least two fat neighbors, 
then the special graph $\s(\Ho)$ of $\Ho$ 
is isomorphic to 
$\mathcal{Q}_{0,0,1}$, $\mathcal{Q}_{1,0,1}$, or $\mathcal{Q}_{0,1,1}$. 
\item[{\rm (ii)}]
If every slim vertex of $\Ho$ has exactly one fat neighbor, 
then the special graph $\s(\Ho)$ of $\Ho$ 
is isomorphic to 
$\mathcal{Q}_{p,q,r}$ 
for some non-neg\-a\-tive integers $p,q,r$ with 
$p+q \leq r$ 
or 
one of the $15$ edge-signed graphs in Figure~\ref{fig:004}. 
\end{itemize}
\end{theorem}

\begin{proof}
The statement (i) follows from Lemma~\ref{lm:004}. 
We show (ii). 
Suppose that every slim vertex of $\Ho$ has exactly one
fat neighbor. By Lemma~\ref{lm:007}, $\s(\Ho)$ is connected,
and by Lemma~\ref{cor:MS-B}, $\s(\Ho)$ has smallest
eigenvalue at least $-\tau$. Moreover, by Lemma~\ref{lm:012},
$\s(\Ho)$ does not contain an induced edge-signed subgraph
isomorphic to $\mathcal{T}_1$. Now Theorem~\ref{p:005++}
implies that $\s(\Ho)$ is isomorphic $\mathcal{Q}_{p,q,r}$
or one of the $15$ edge-signed graphs in Figure~\ref{fig:004}.
\end{proof}

\begin{cor}\label{thm:001-2}
Let $\Ho$ be a fat 
$(-1-\tau)$-ir\-re\-duc\-i\-ble 
Hoffman graph. 
Then 
the special graph $\s(\Ho)$ of $\Ho$ 
is isomorphic to 
$\mathcal{Q}_{0,0,1}$, $\mathcal{Q}_{1,0,1}$, $\mathcal{Q}_{0,1,1}$, 
or one of the $15$ edge-signed graphs 
in Figure~\ref{fig:004}. 
\end{cor}

\begin{proof}
Since $\Ho$ is $(-1-\tau)$-ir\-re\-duc\-i\-ble, 
$\Ho$ is indecomposable. 
If some slim vertex of $\Ho$ has at least two fat neighbors, 
then the statement holds by Theorem \ref{thm:001-1} (i). 
Suppose that 
every slim vertex of $\Ho$ has exactly one fat neighbor. 
By Theorem \ref{thm:001-1} (ii), 
$\s(\Ho)$ is isomorphic to 
$\mathcal{Q}_{p,q,r}$ 
for some non-neg\-a\-tive integers $p,q,r$, 
or one of the $15$ edge-signed graphs 
in Figure~\ref{fig:004}. 
Since $\Ho$ is $(-1-\tau)$-ir\-re\-duc\-i\-ble, 
the former 
case occurs only for $r=1$ by Lemma~\ref{p:4-001}. 
Hence the corollary holds. 
\end{proof}

\subsection{The classification of fat 
Hoffman graphs with smallest eigenvalue at least 
$-1-\tau$}

\begin{figure}[!h]
\begin{center}
\begin{tabular}{cccc}
    \includegraphics[scale=0.15]{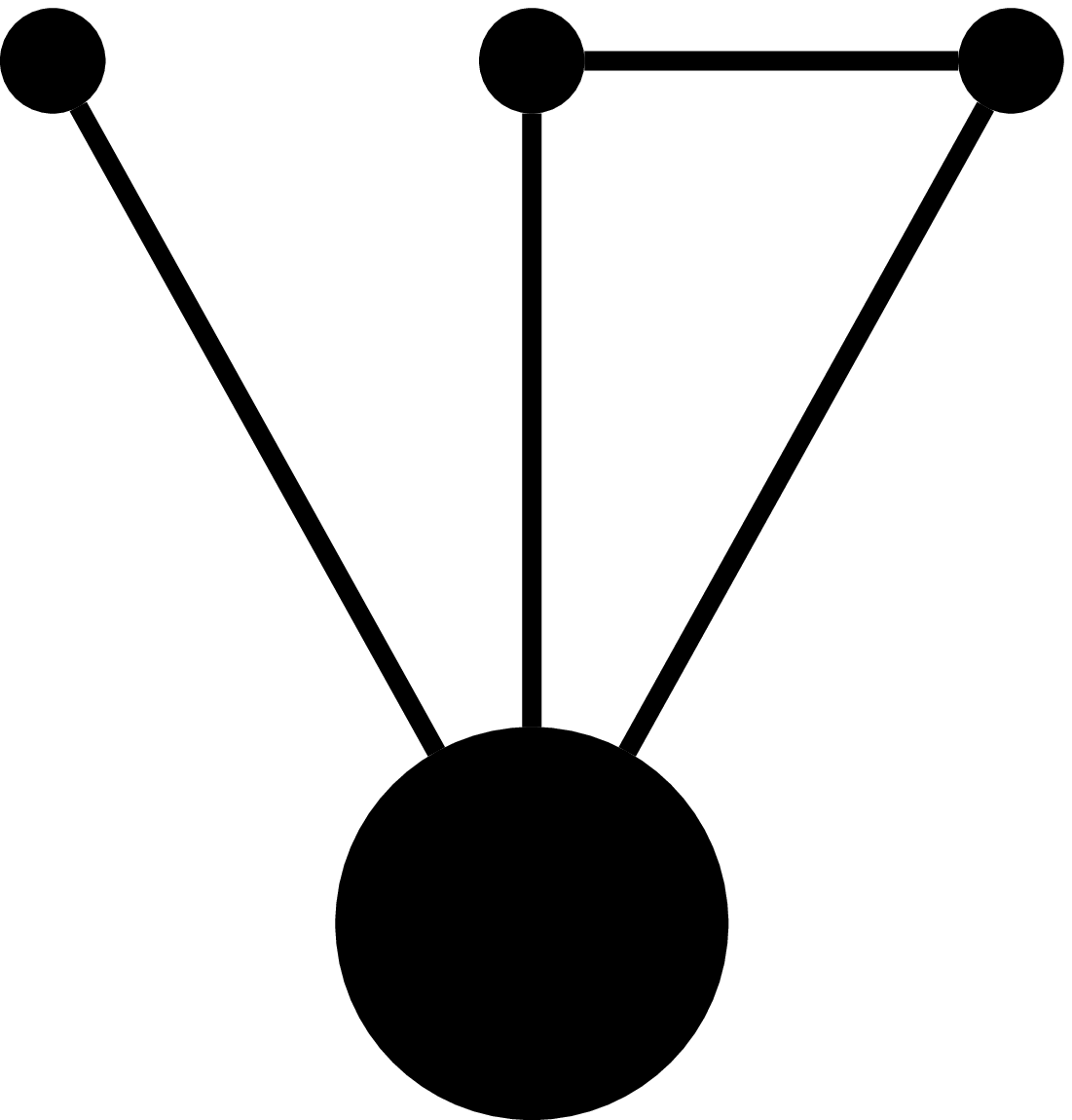} &
    \includegraphics[scale=0.15]{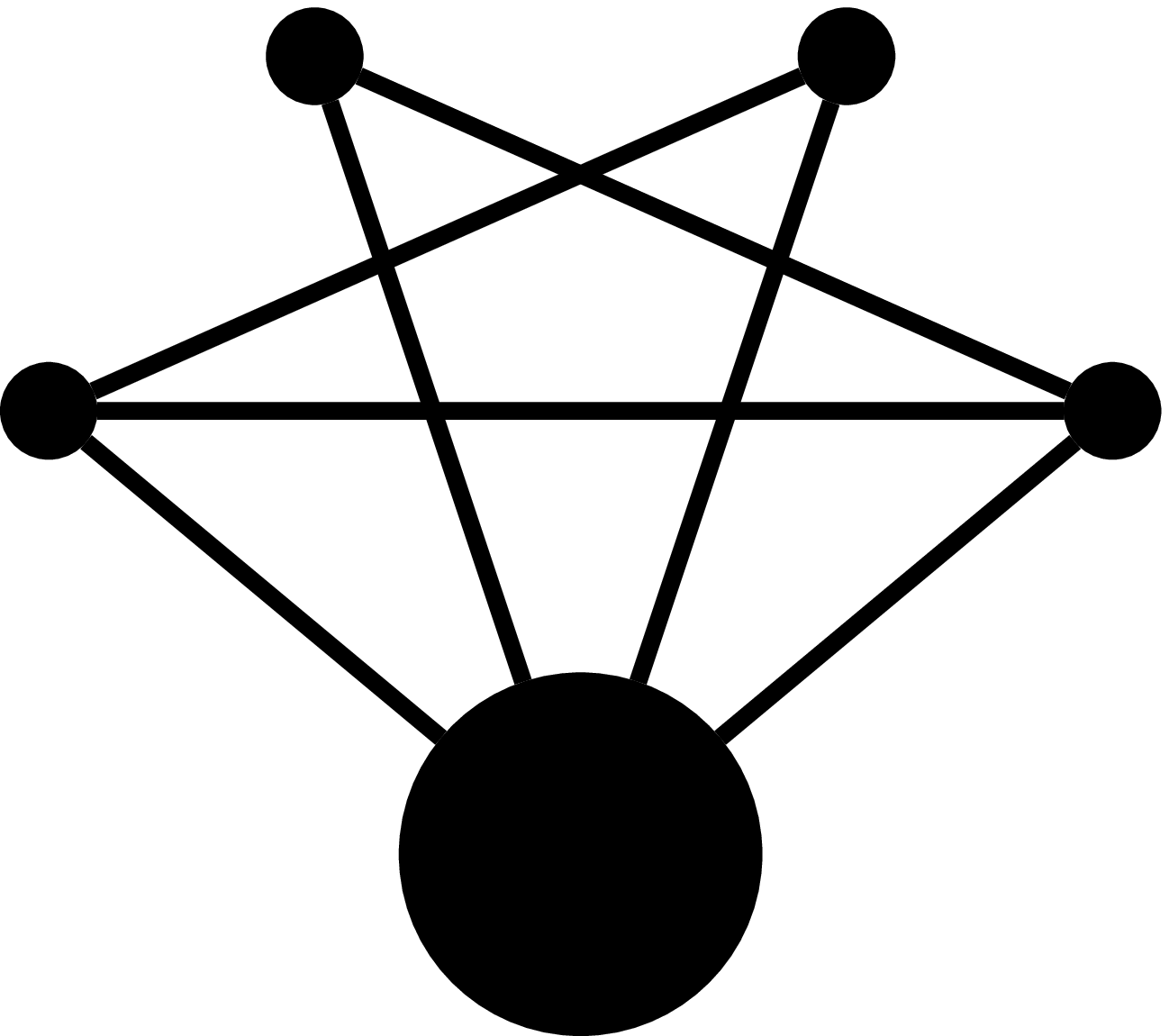} &
    \includegraphics[scale=0.15]{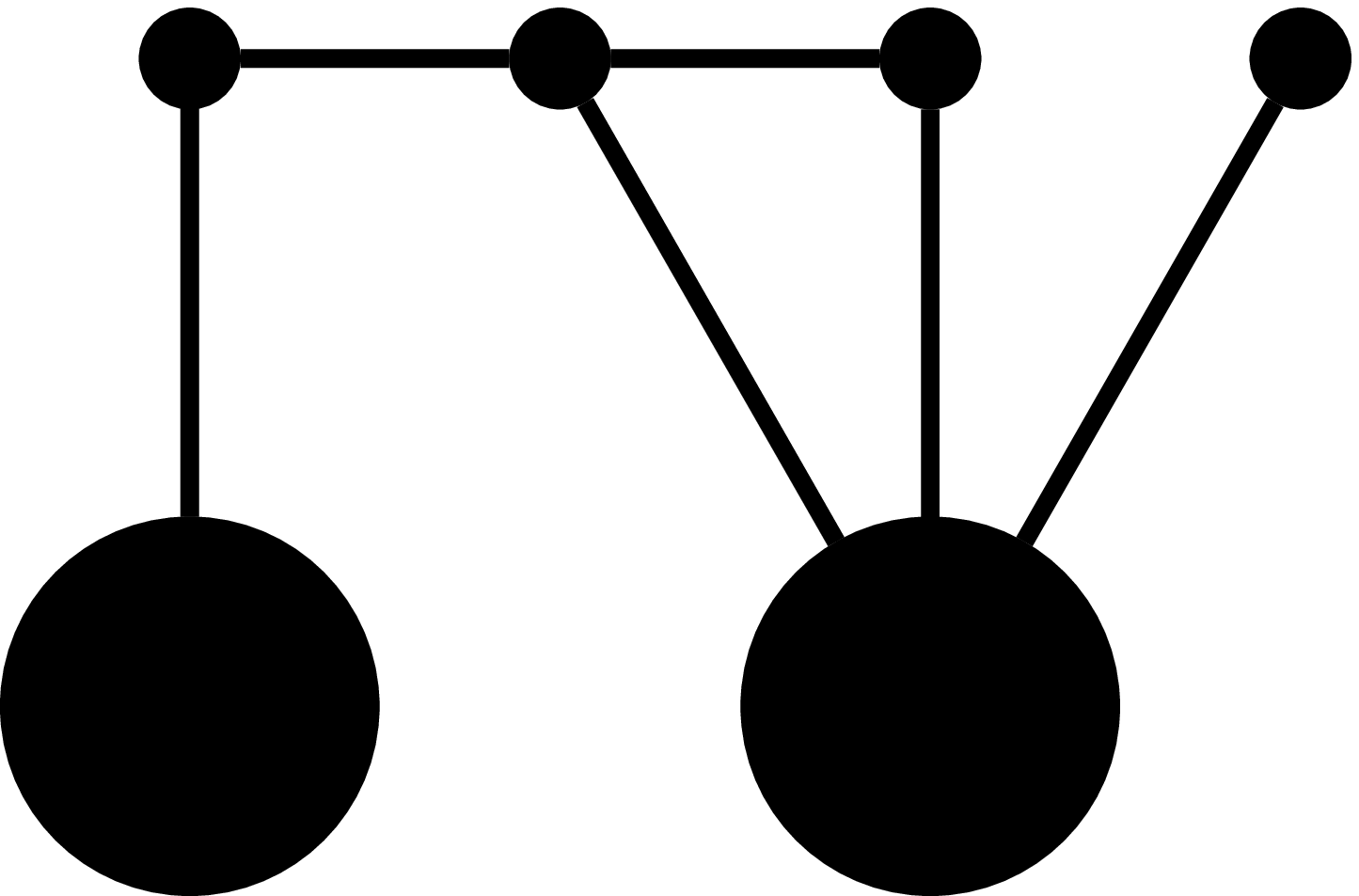} &
    \includegraphics[scale=0.15]{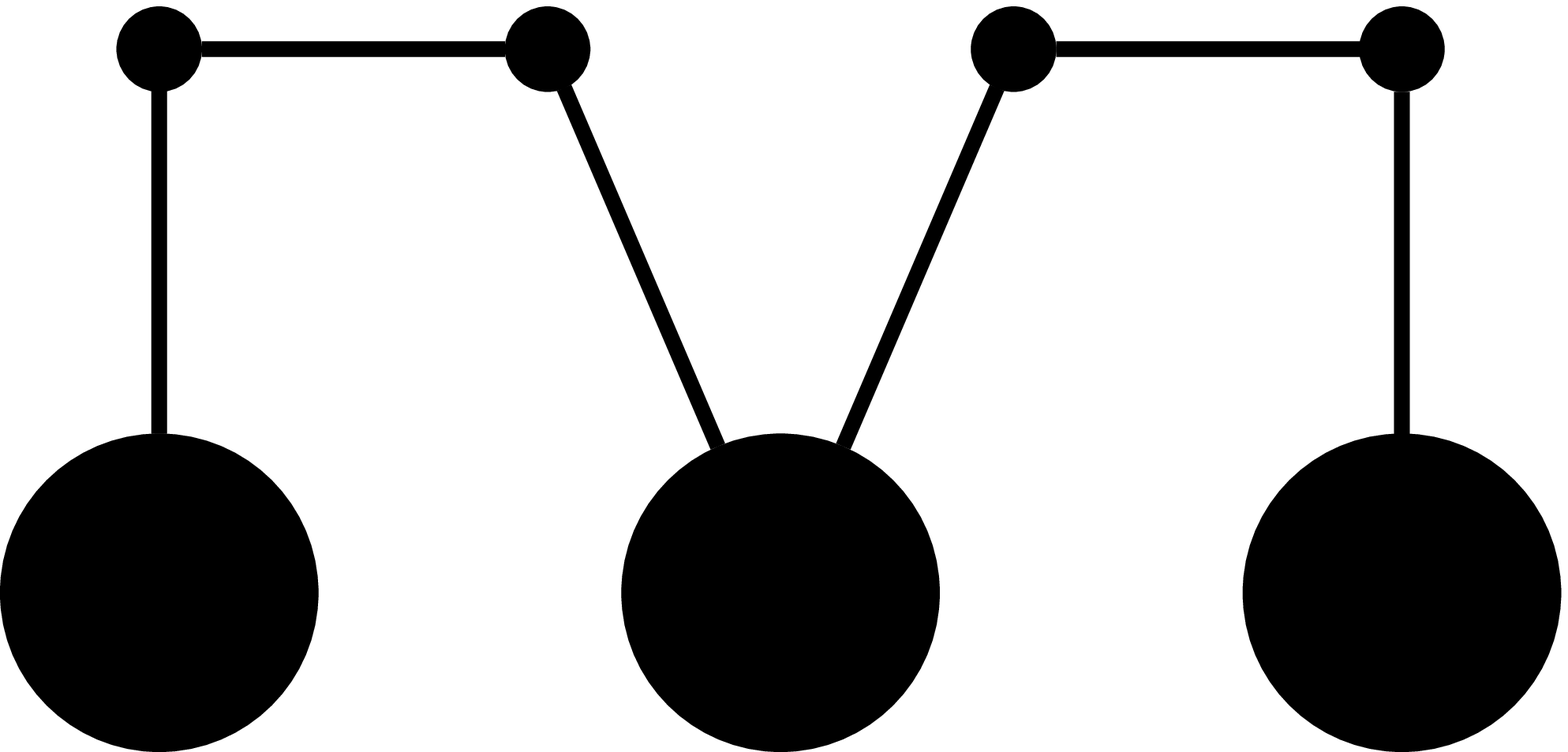} \\
    $\Ho_{3,1}^1$ & $\Ho_{4,1}^1$ &
    $\Ho_{4,2}^1$ & $\Ho_{4,3}^1$ \\
    &&& \\
    &&& \\
    \includegraphics[scale=0.15]{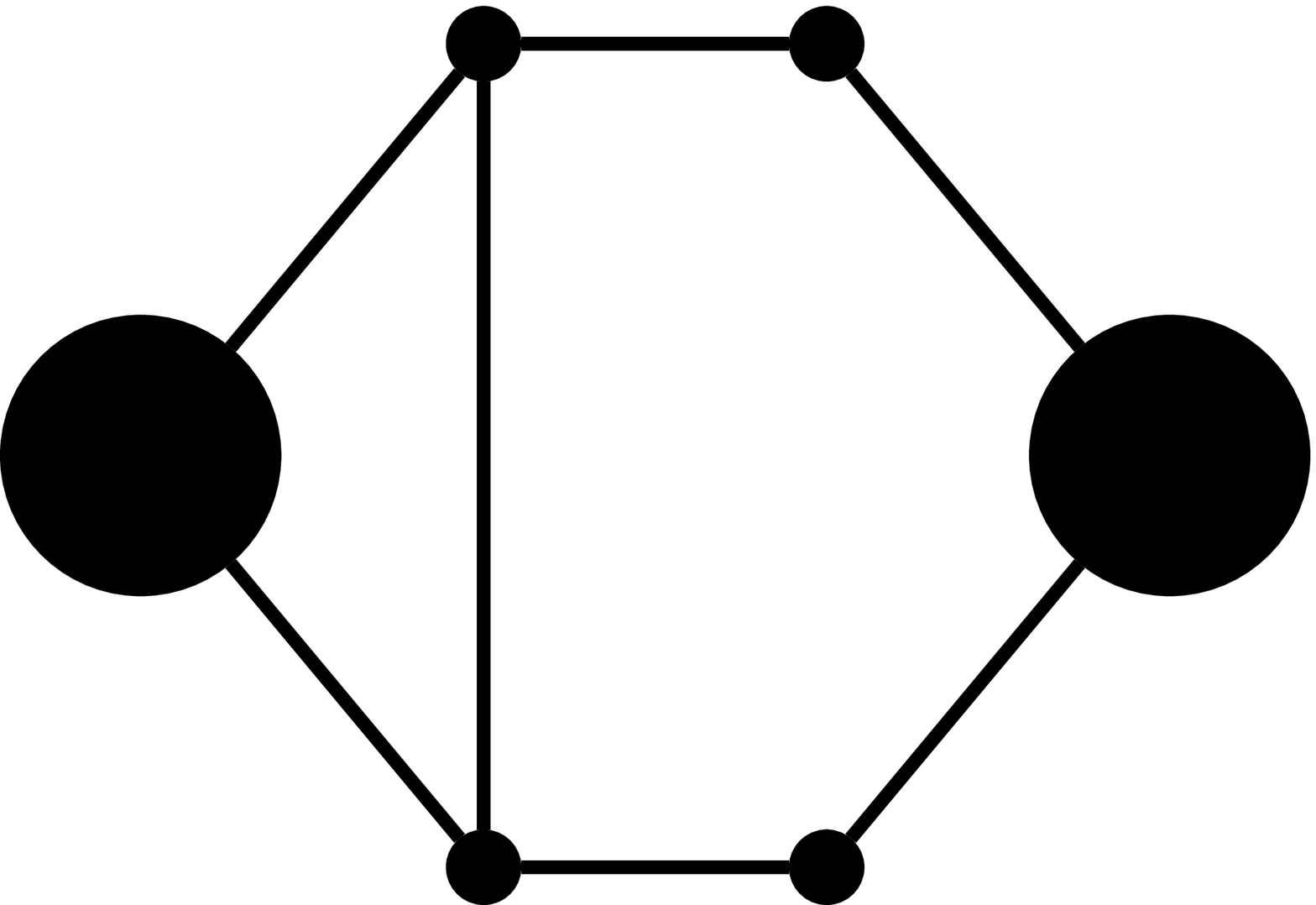} &
    \includegraphics[scale=0.15]{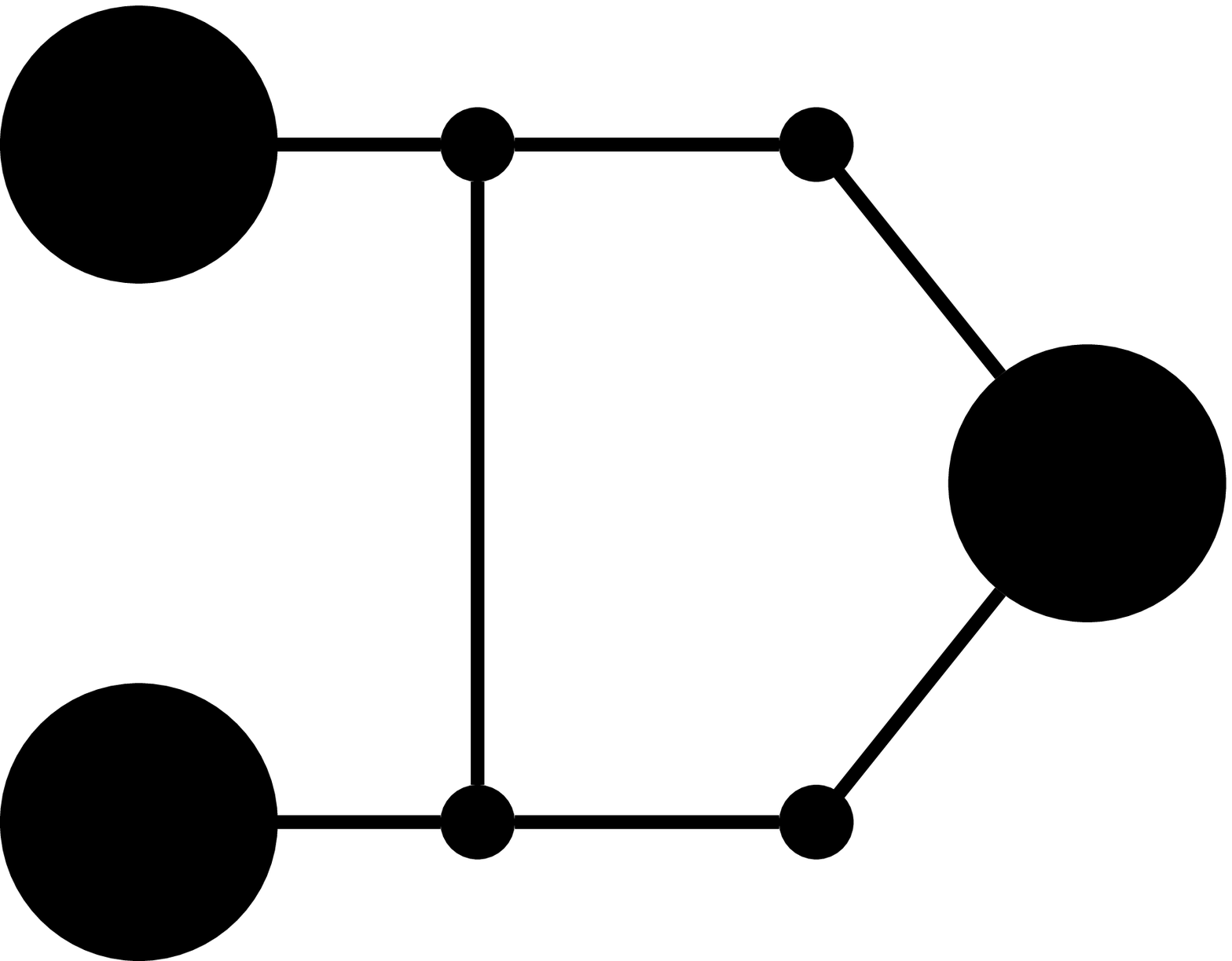} &
    \includegraphics[scale=0.15]{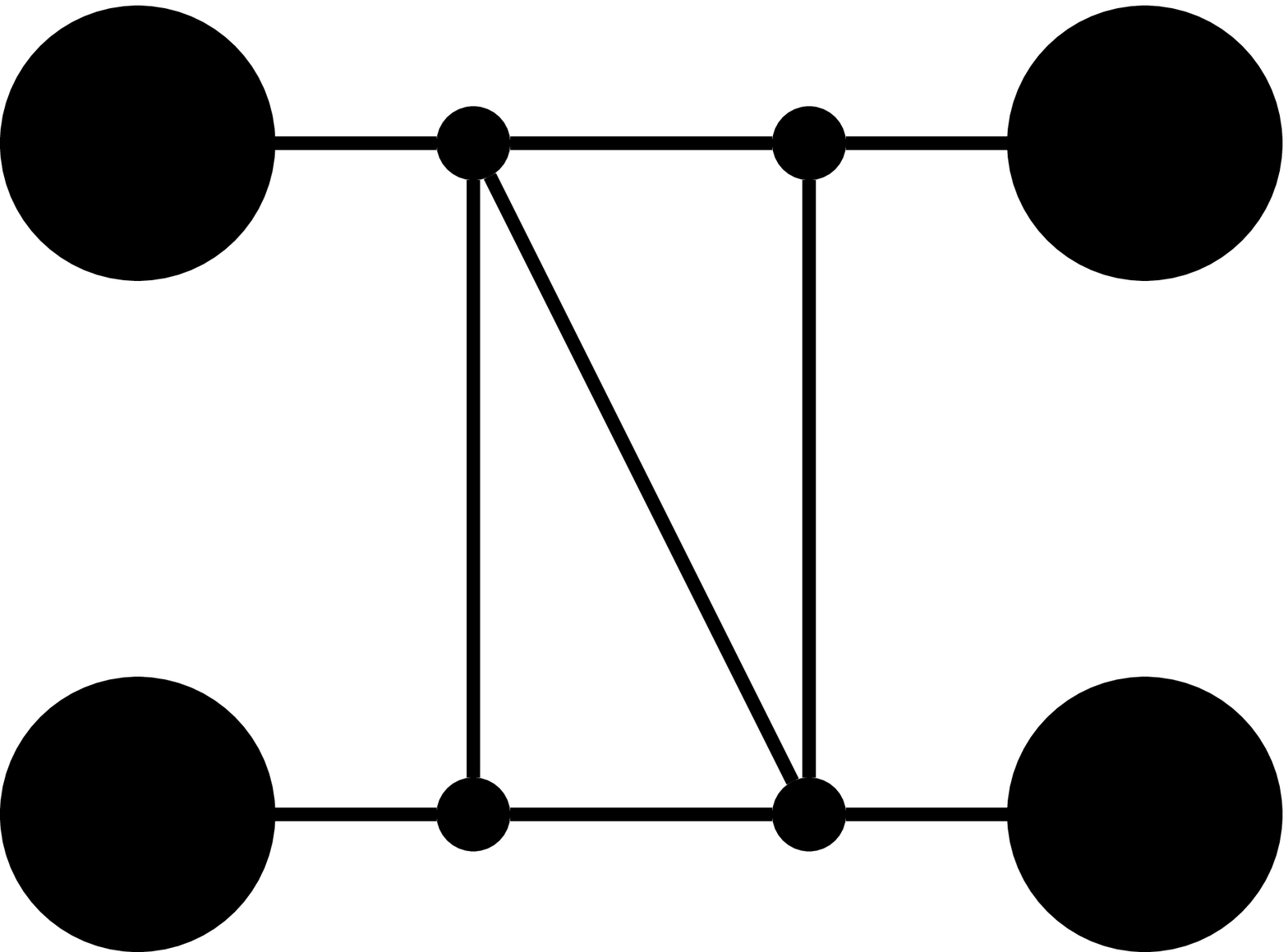} &
    \includegraphics[scale=0.15]{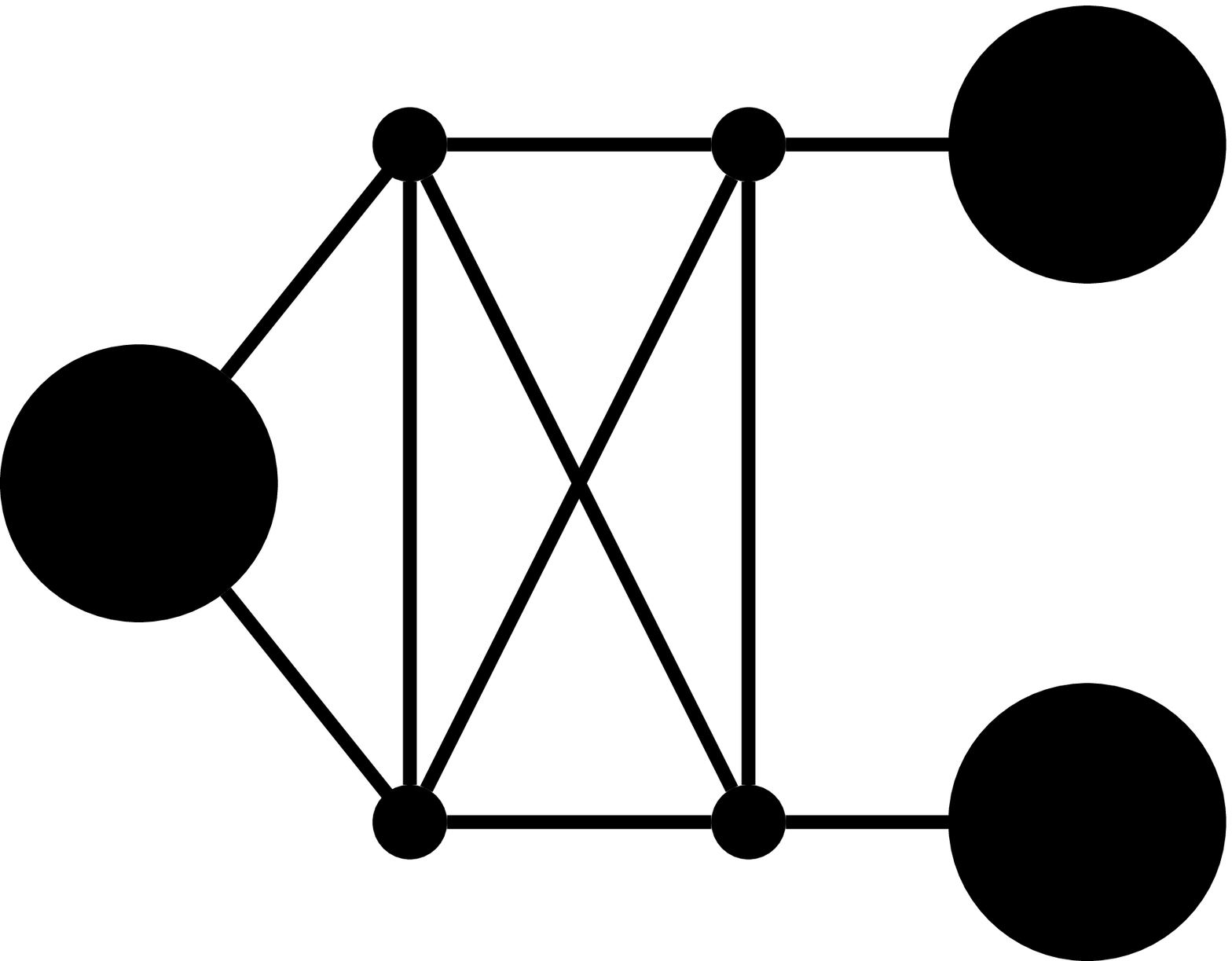} \\
$\Ho_{4,3}^2$ & $\Ho_{4,4}^1$ & 
$\Ho_{4,5}^1$ & $\Ho_{4,5}^2$ \\
\end{tabular}
\end{center}
\caption{The fat 
$(-1-\tau)$-ir\-re\-duc\-i\-ble 
Hoffman graphs with $3$ or $4$ slim vertices}
\label{fig:FHG-4}
\end{figure}

\begin{figure}[!h]
\begin{center}
\begin{tabular}{ccc}
    \includegraphics[scale=0.15]{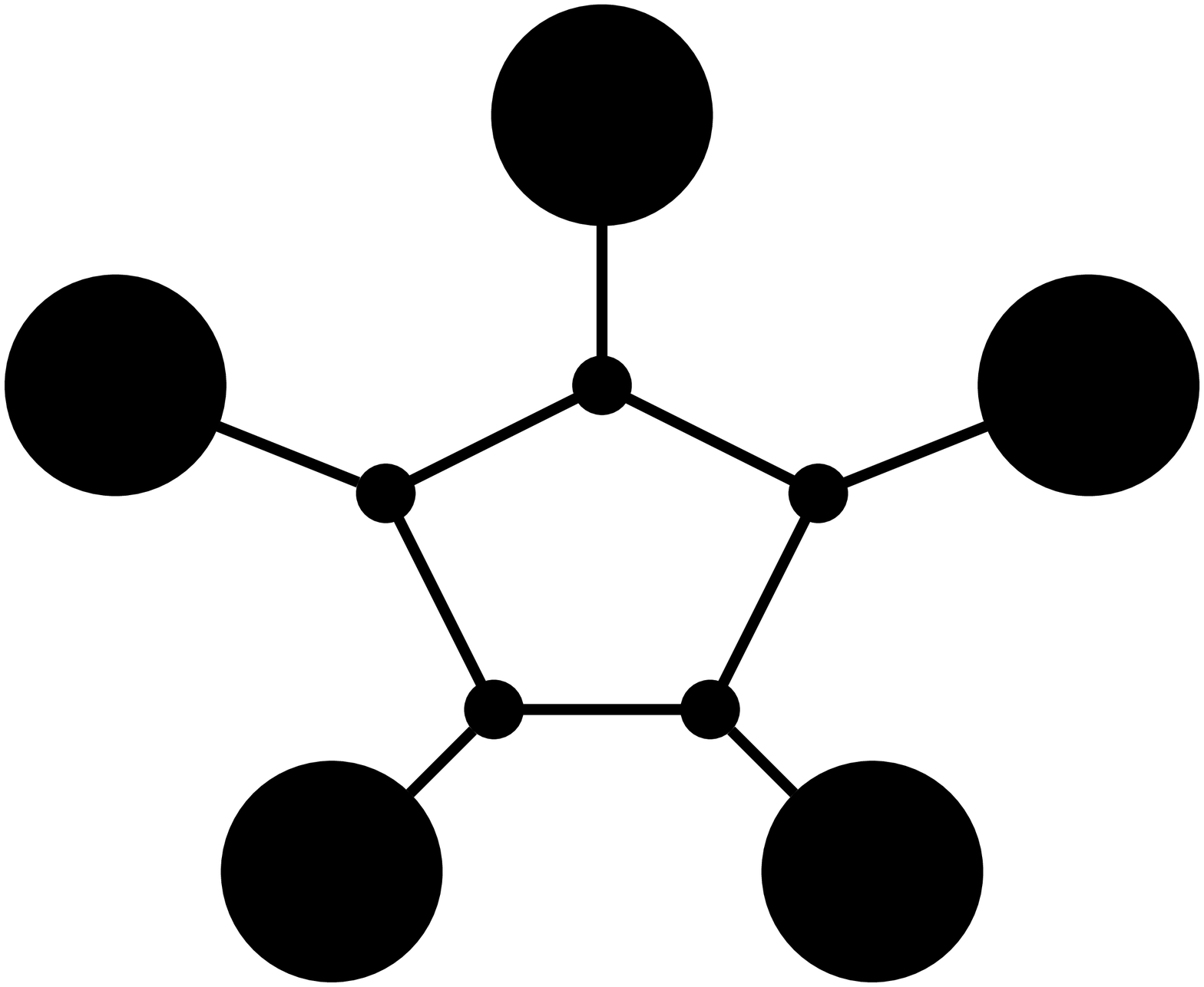} &
    \includegraphics[scale=0.15]{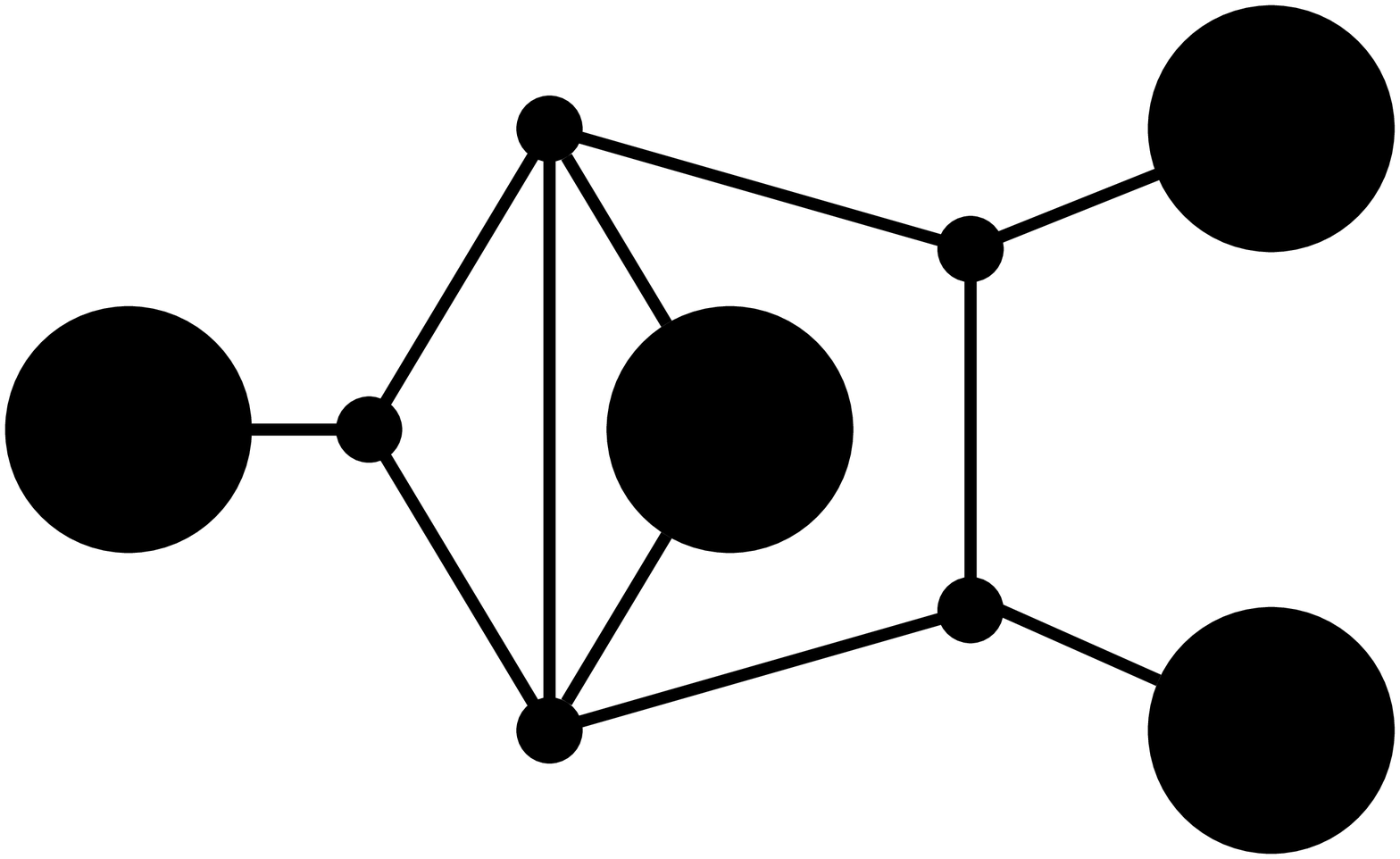} &
    \includegraphics[scale=0.15]{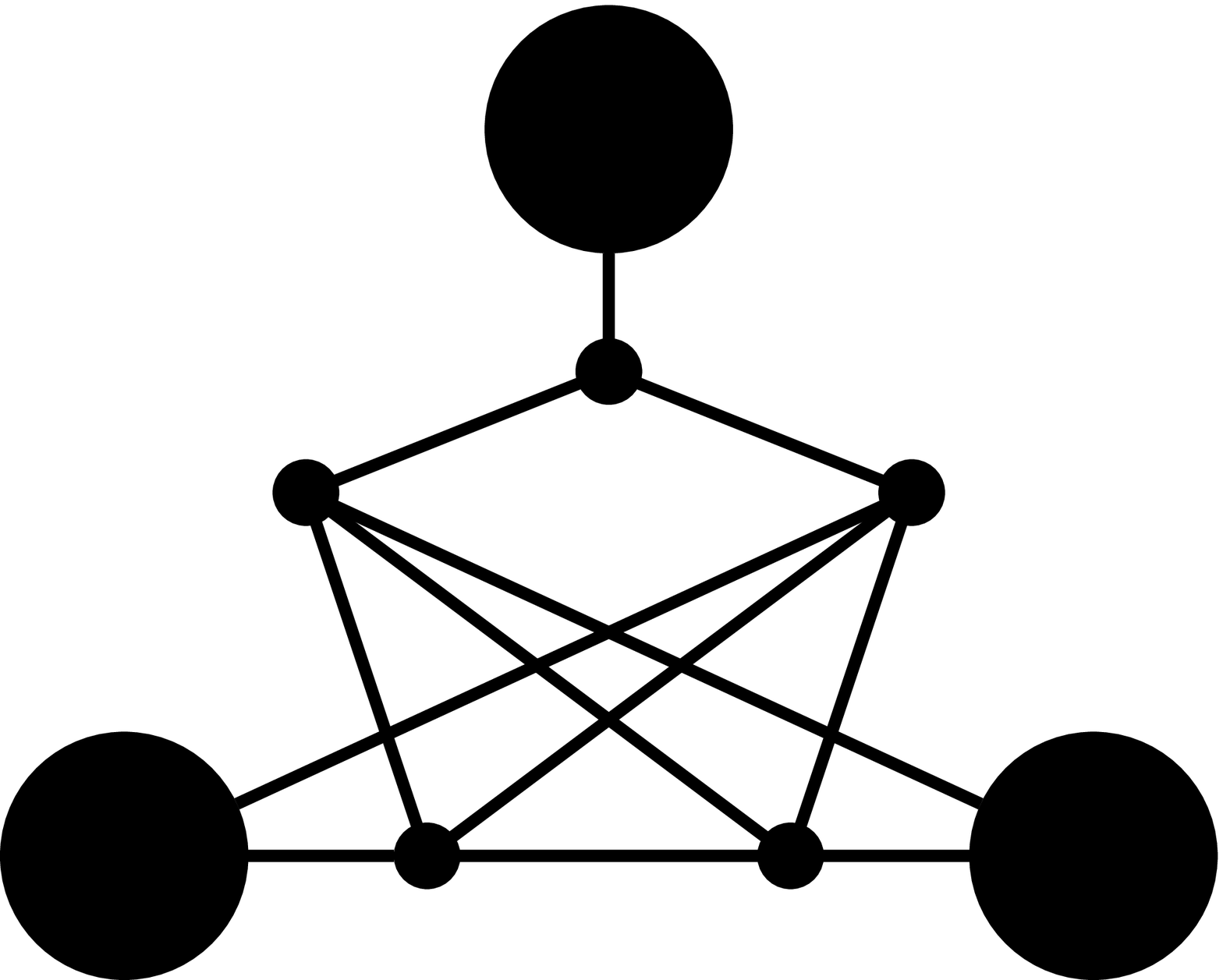} \\
    $\Ho_{5,1}^1$ & $\Ho_{5,1}^2$ & $\Ho_{5,1}^3$ \\
    && \\
    \includegraphics[scale=0.15]{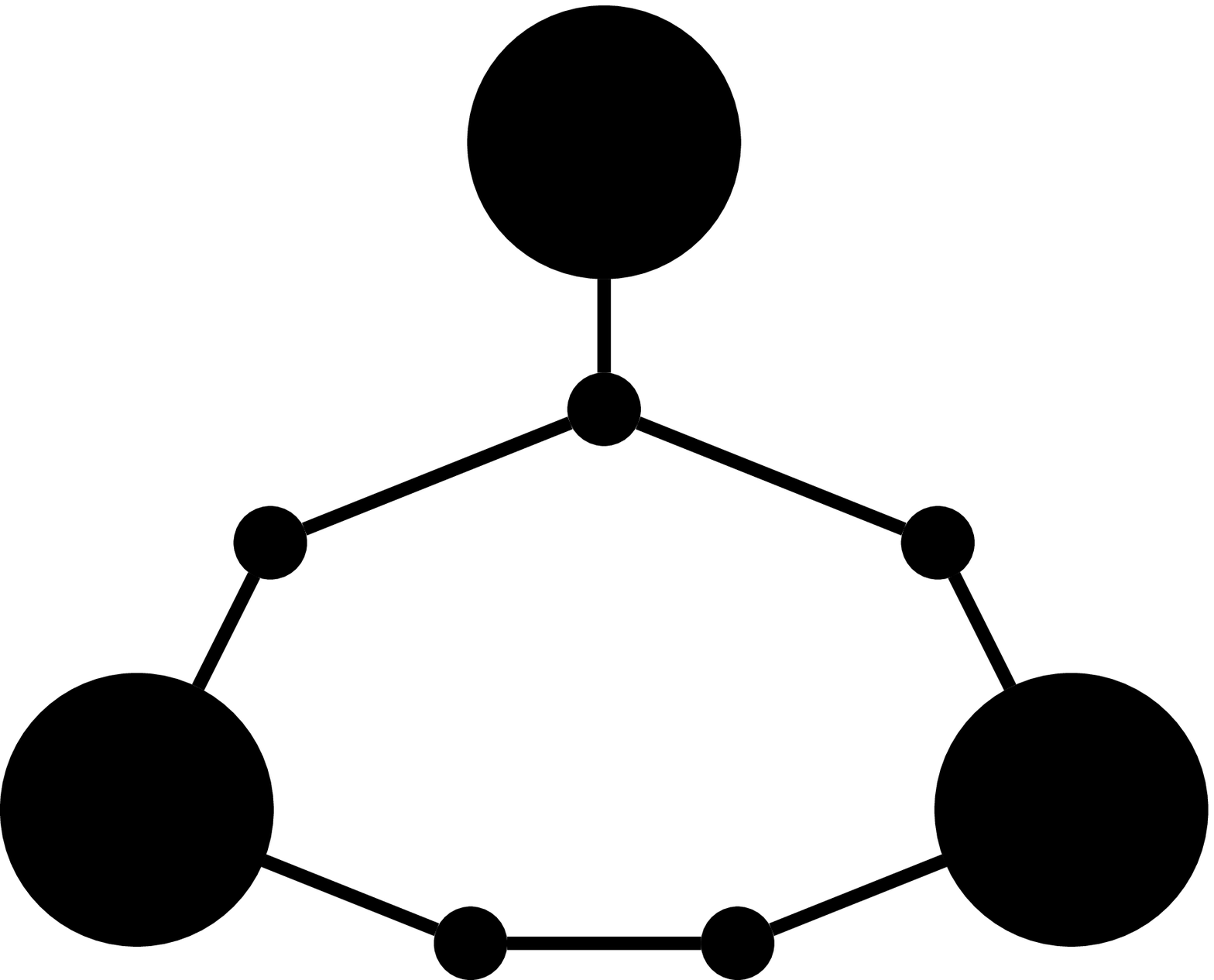} &
    \includegraphics[scale=0.15]{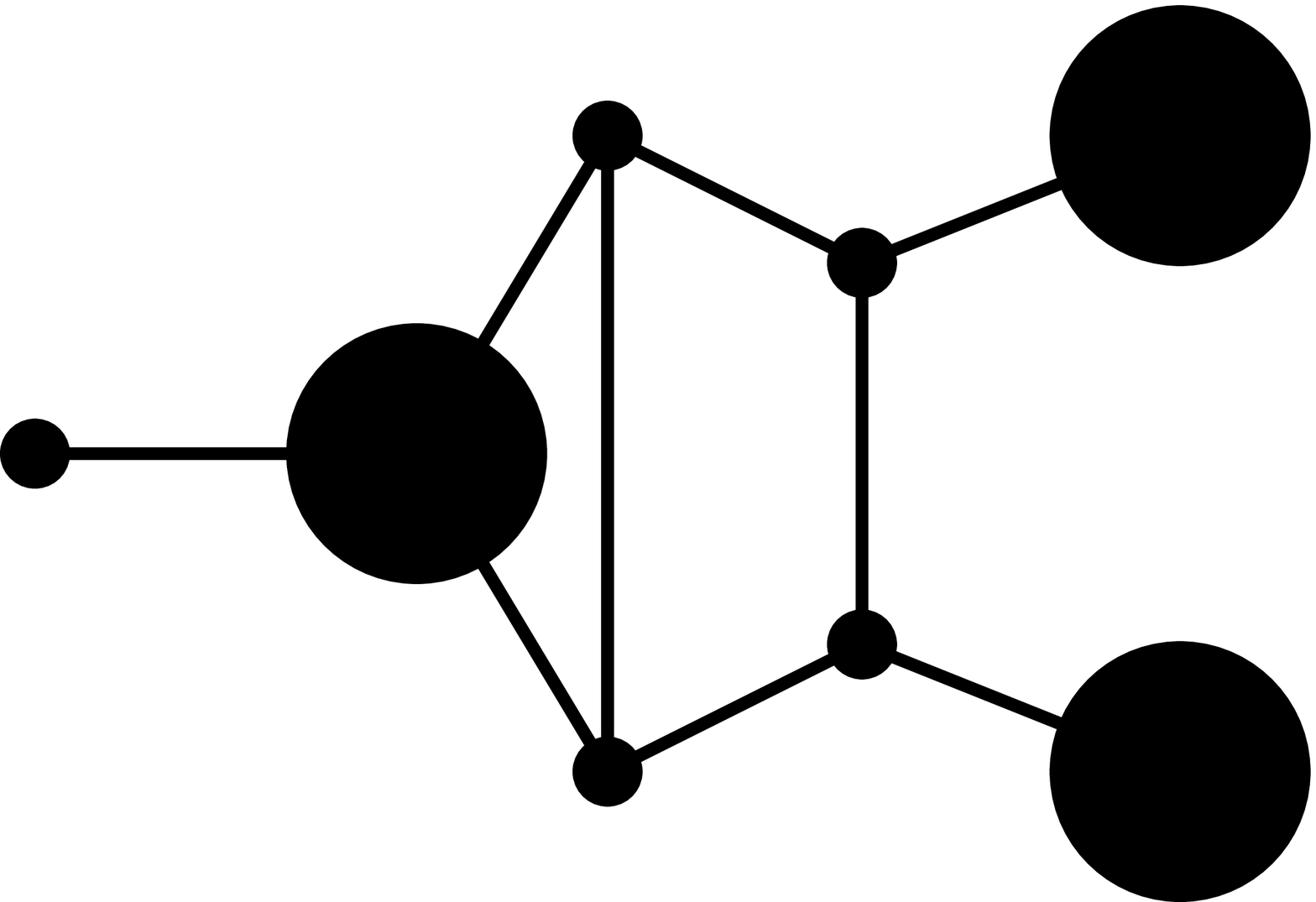} &
    \includegraphics[scale=0.15]{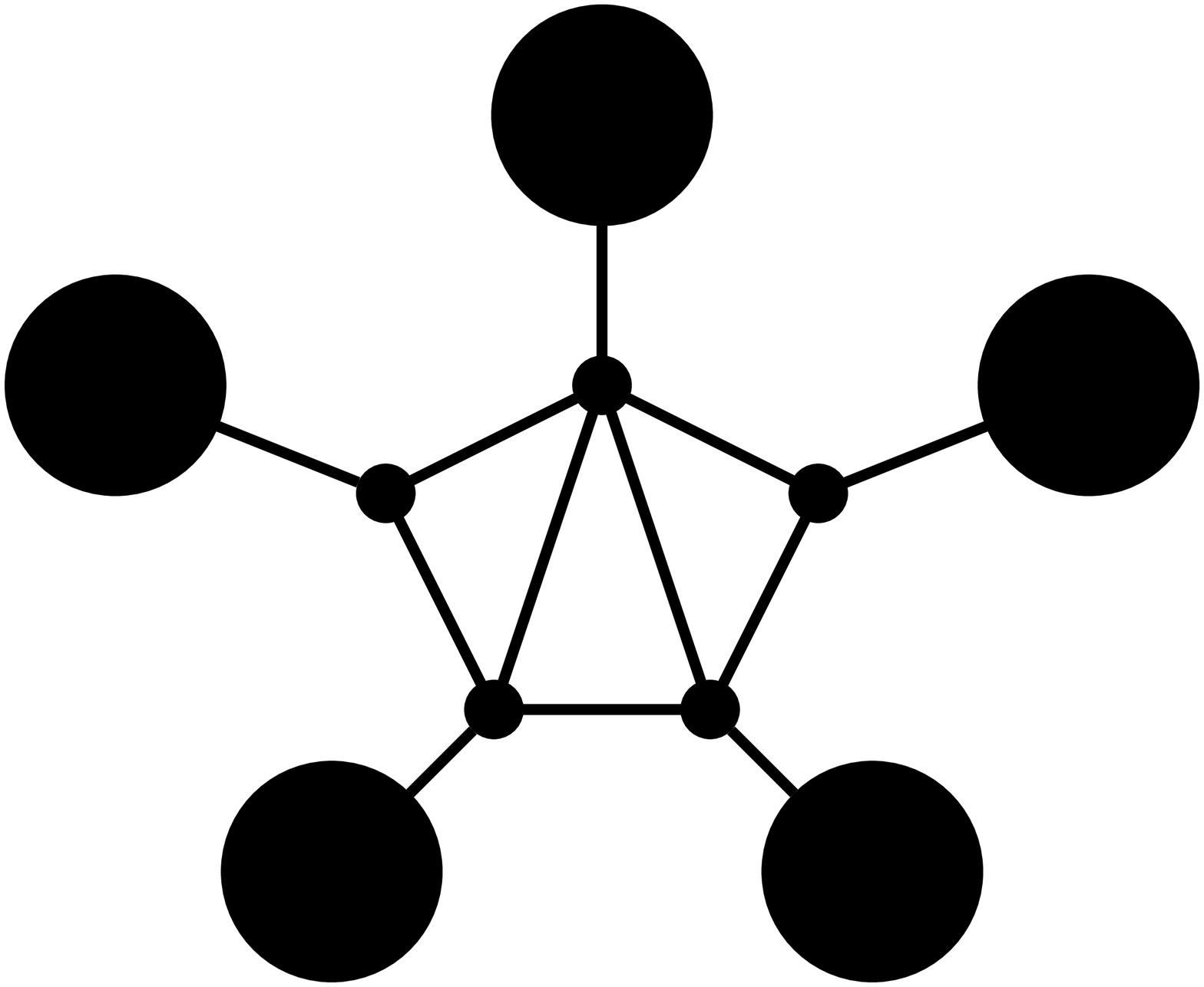} \\
$\Ho_{5,2}^1$ & $\Ho_{5,3}^1$ & $\Ho_{5,4}^1$ \\
    && \\
    \includegraphics[scale=0.15]{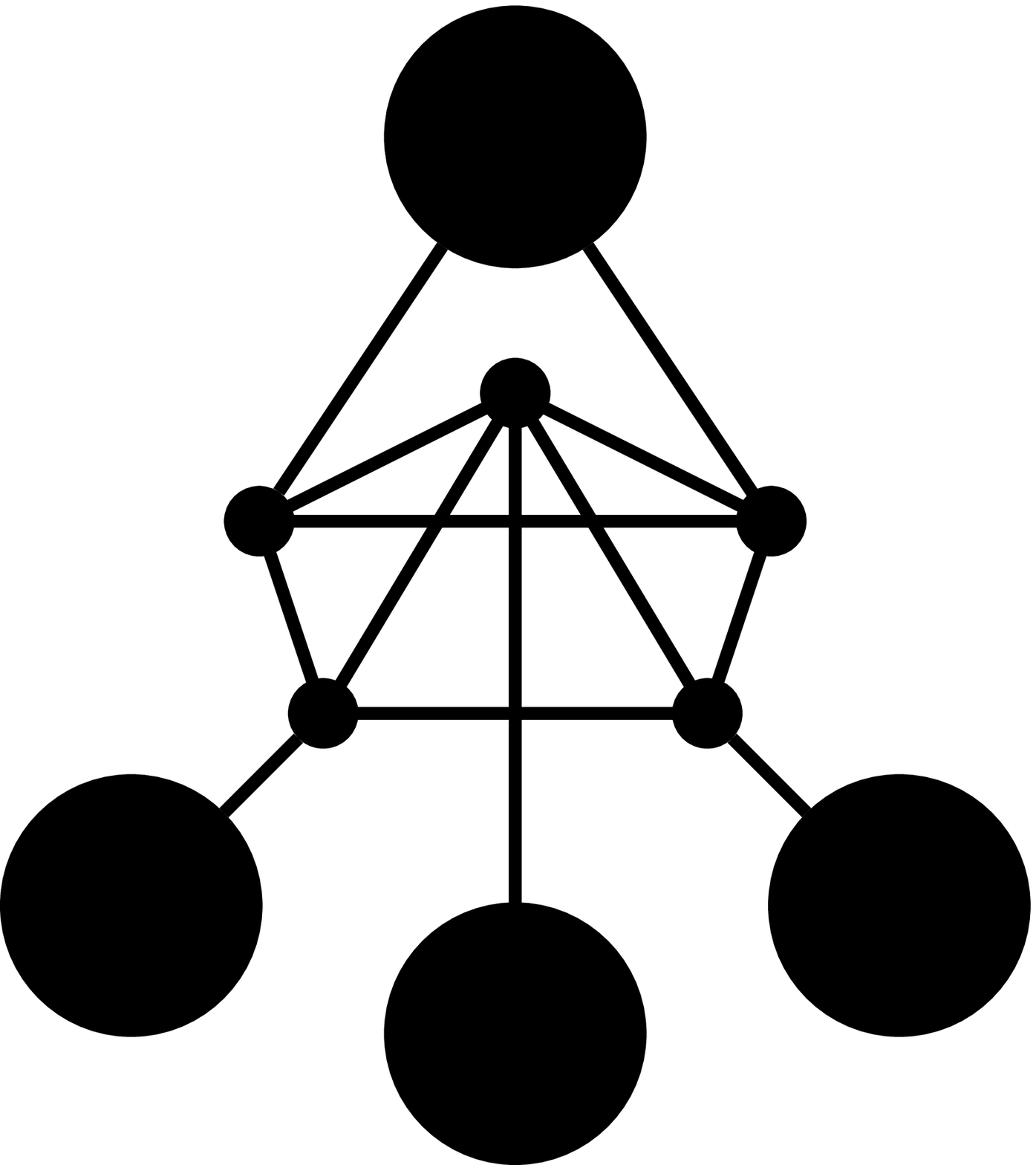} &
    \includegraphics[scale=0.15]{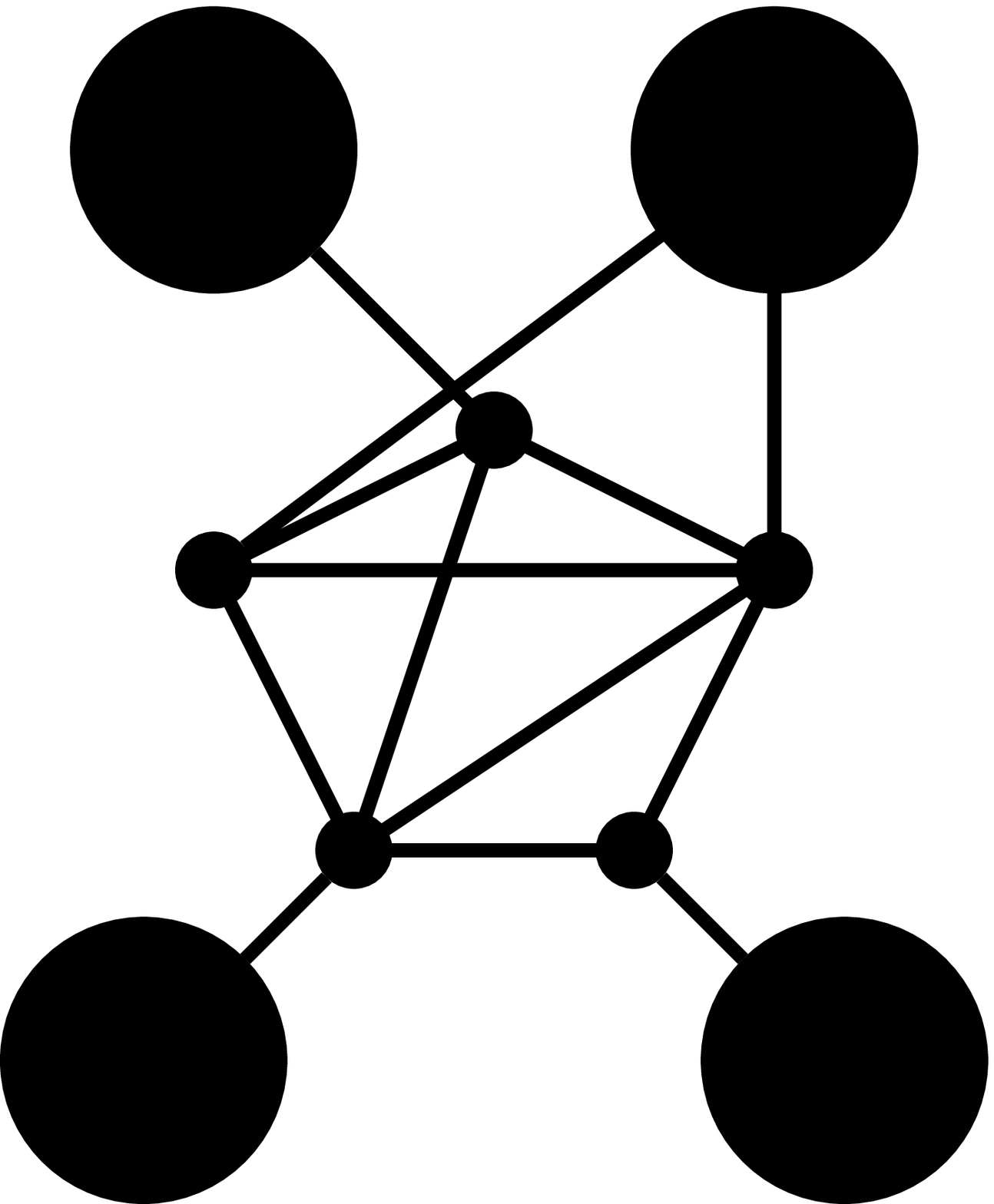} &
    \includegraphics[scale=0.15]{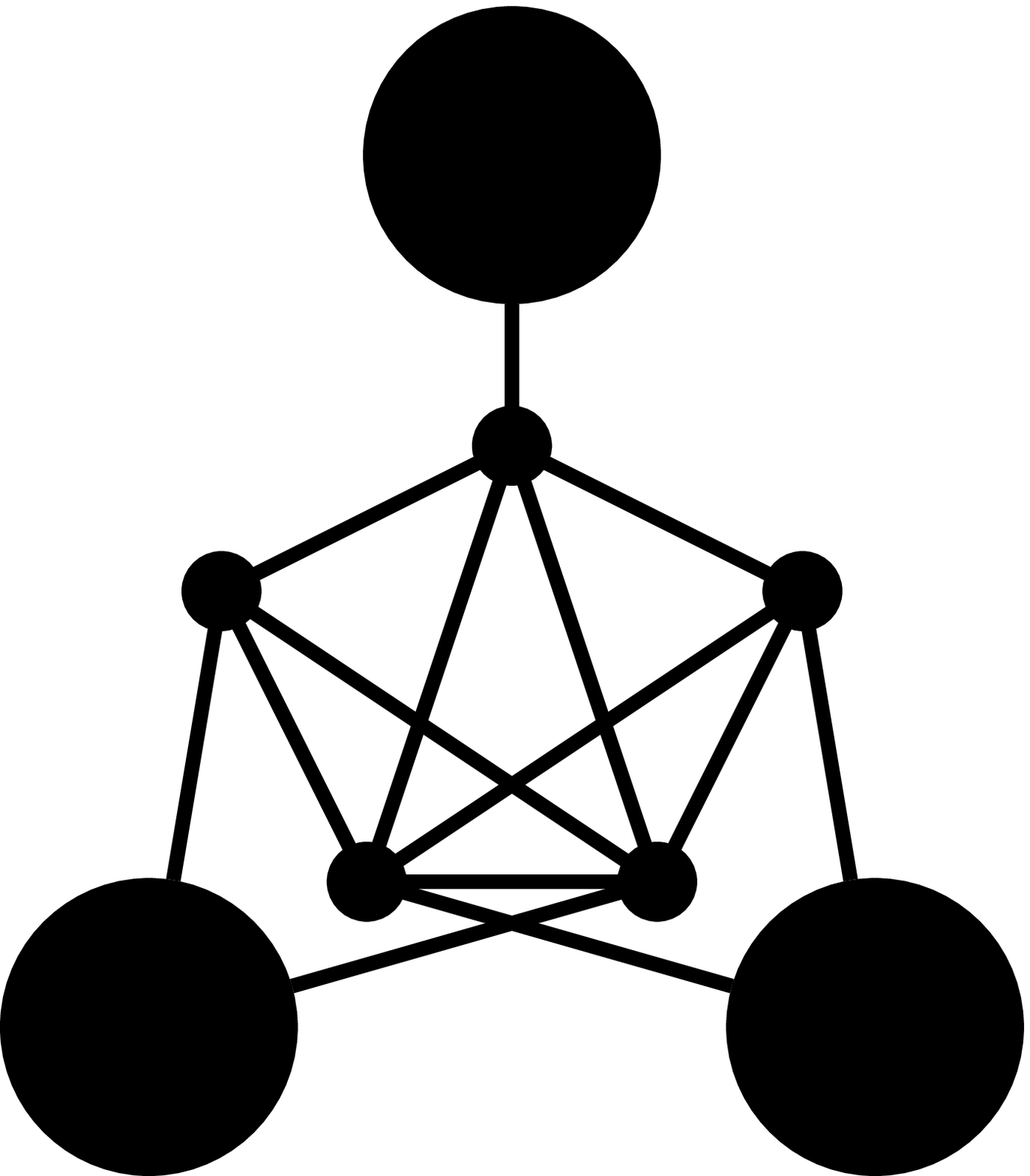} \\
$\Ho_{5,4}^2$ & $\Ho_{5,4}^3$ & $\Ho_{5,4}^4$  \\
    && \\
    \includegraphics[scale=0.15]{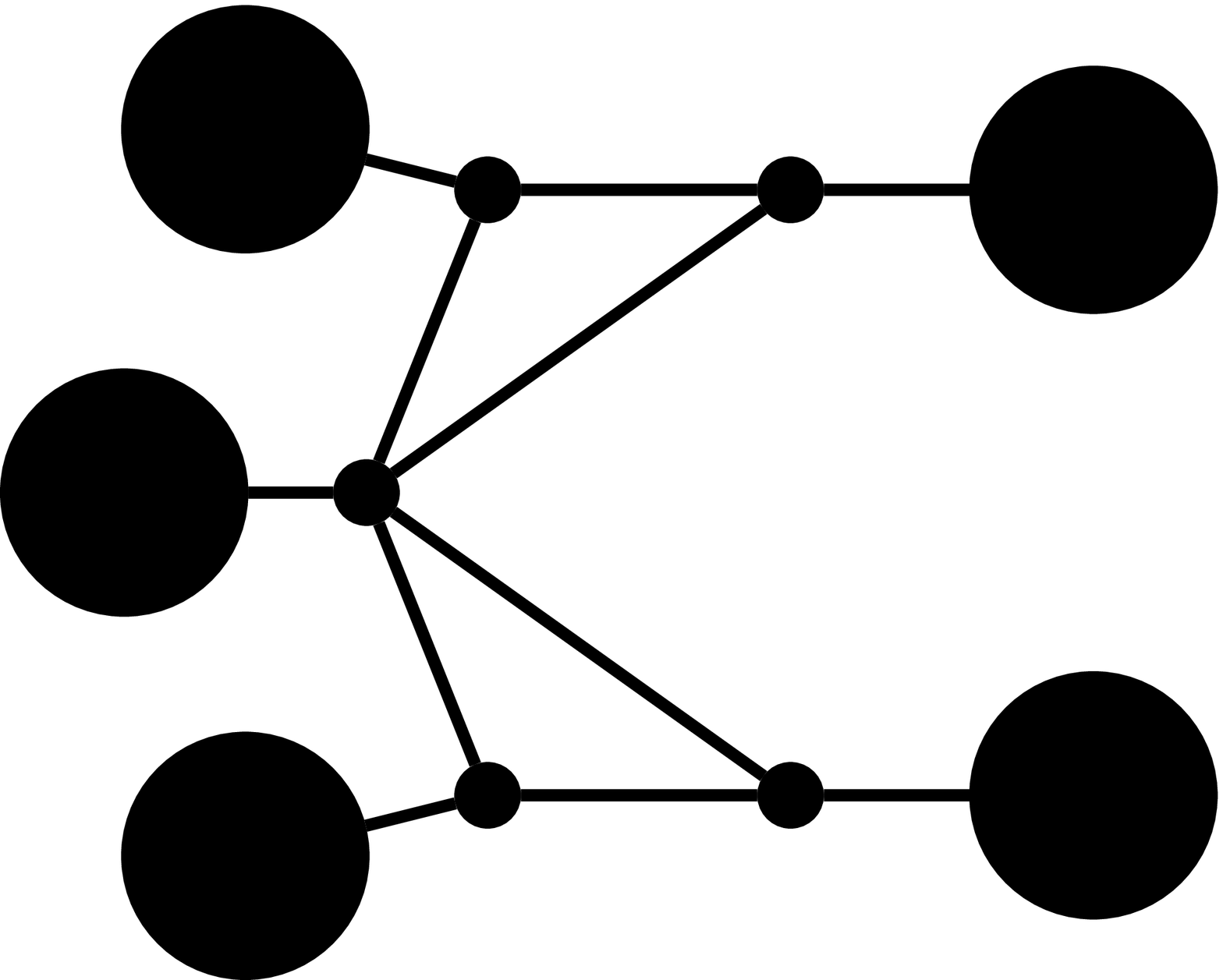} &
    \includegraphics[scale=0.15]{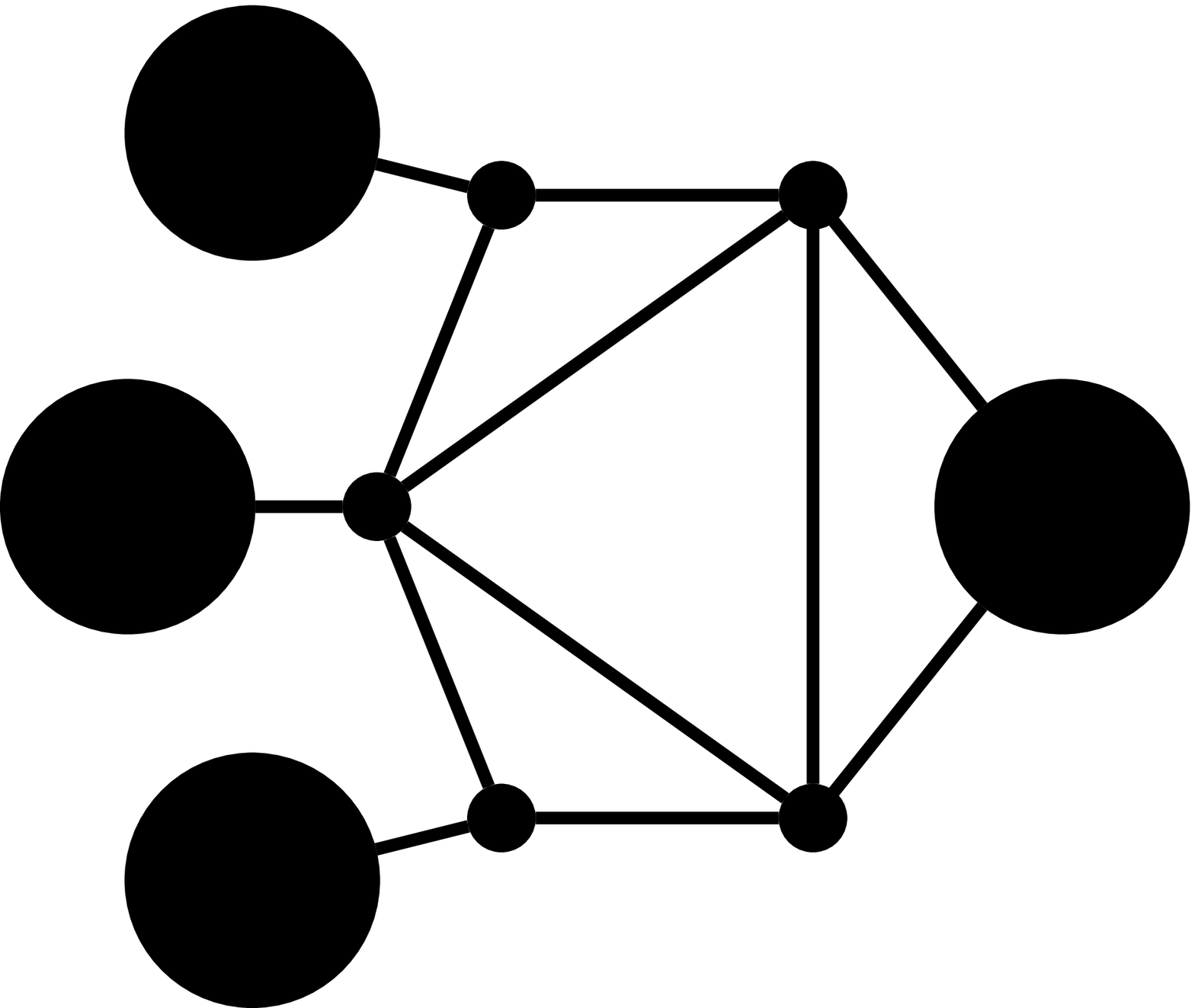} &
    \includegraphics[scale=0.15]{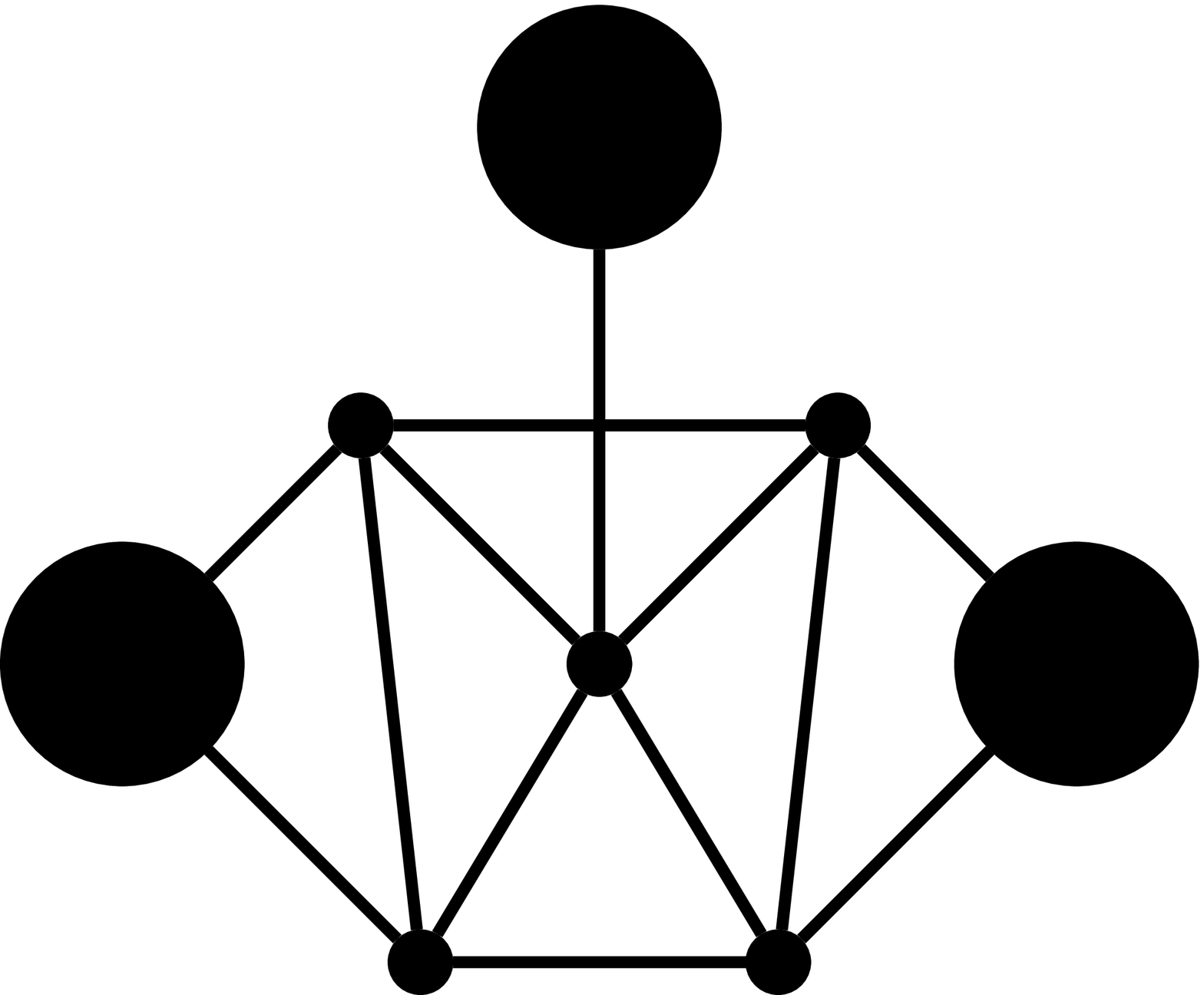} \\
$\Ho_{5,5}^1$ & $\Ho_{5,5}^2$ & $\Ho_{5,5}^3$ \\
    && \\
    \includegraphics[scale=0.15]{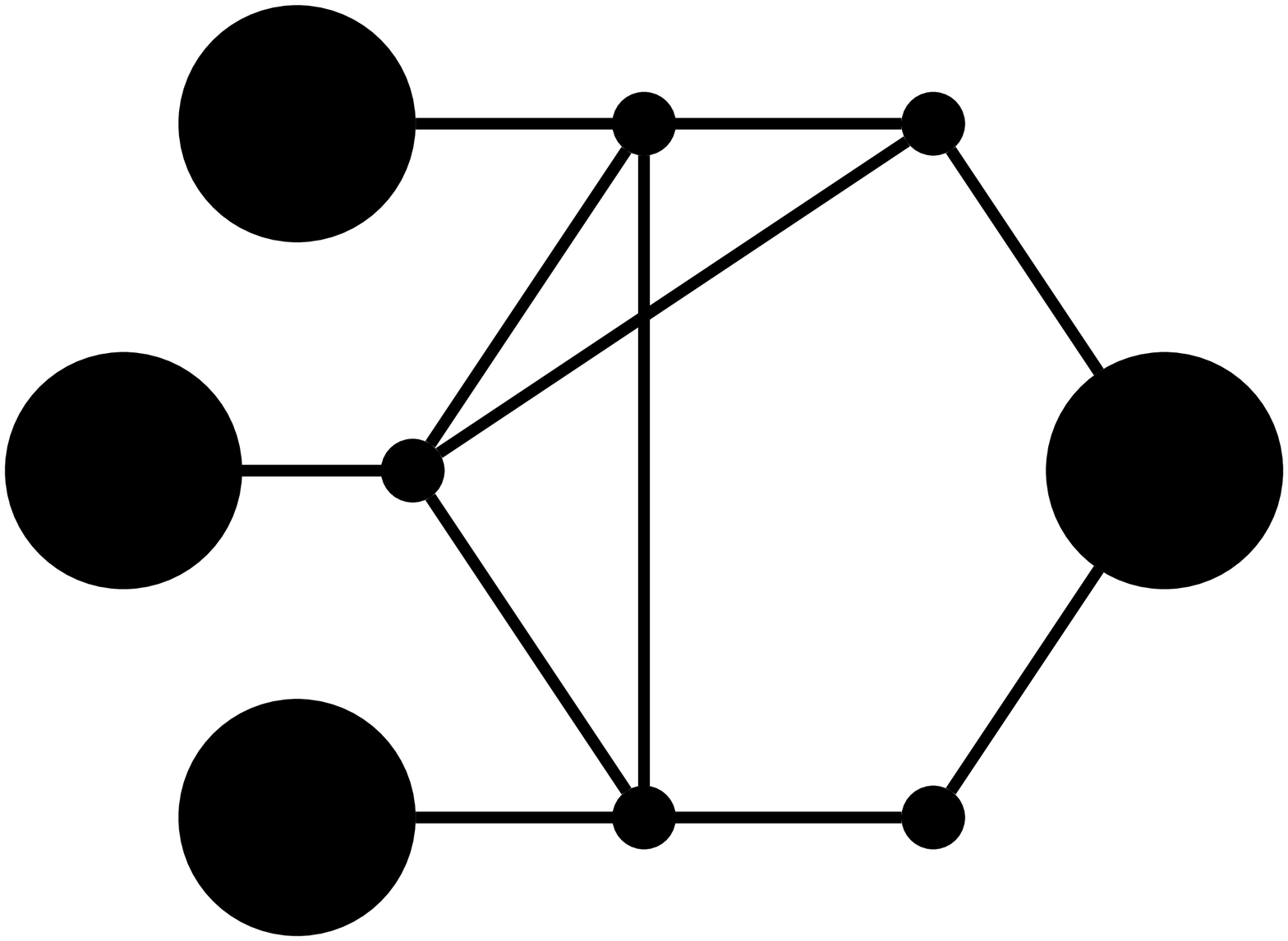} &
 &
 \\
$\Ho_{5,6}^1$ & & \\
\end{tabular}
\end{center}
\caption{The fat 
$(-1-\tau)$-ir\-re\-duc\-i\-ble 
Hoffman graphs with $5$ slim vertices}
\label{fig:FHG-5}
\end{figure}

\begin{figure}[!h]
\begin{center}
\begin{tabular}{ccc}
\includegraphics[scale=0.15]{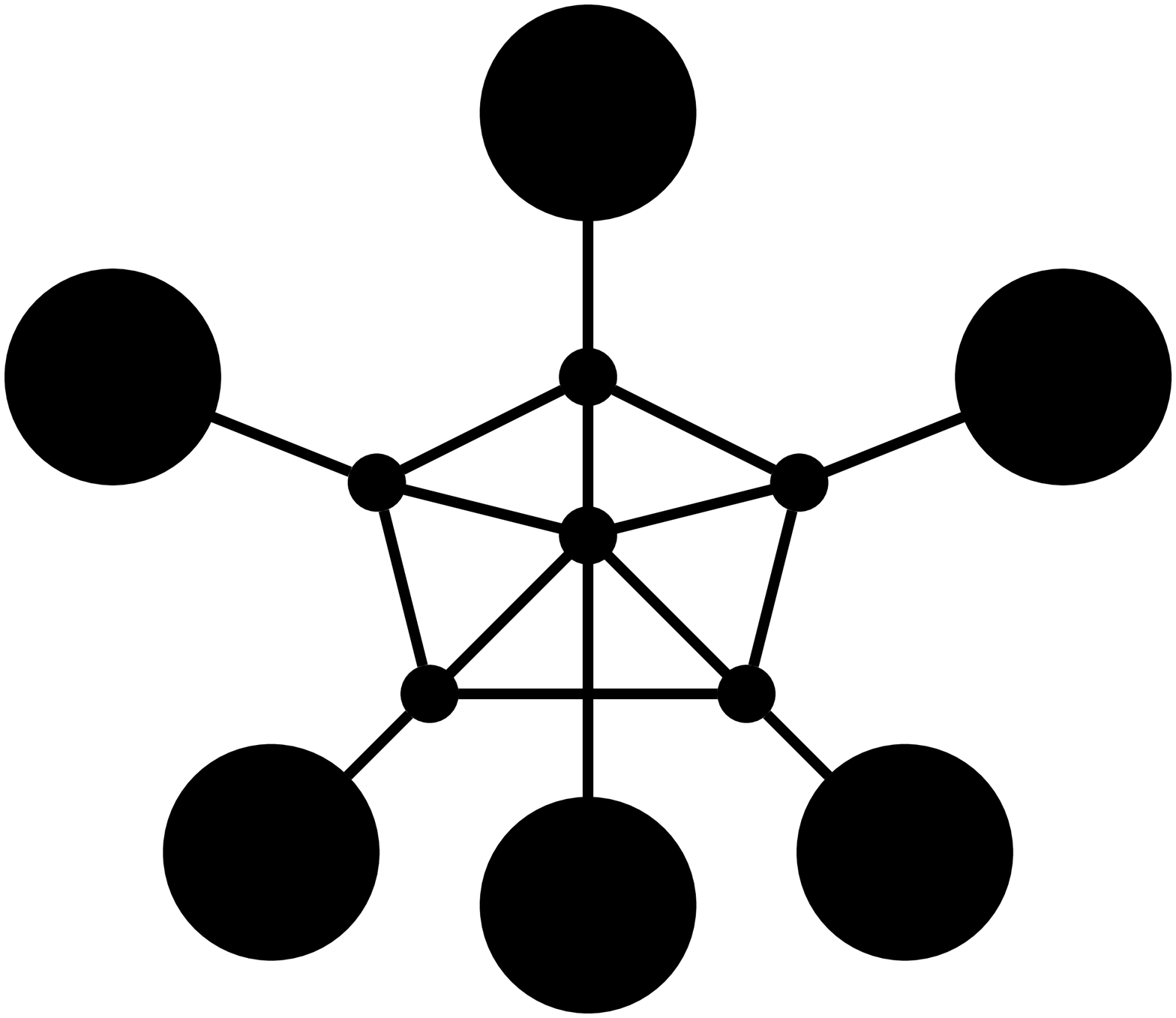} &
\includegraphics[scale=0.15]{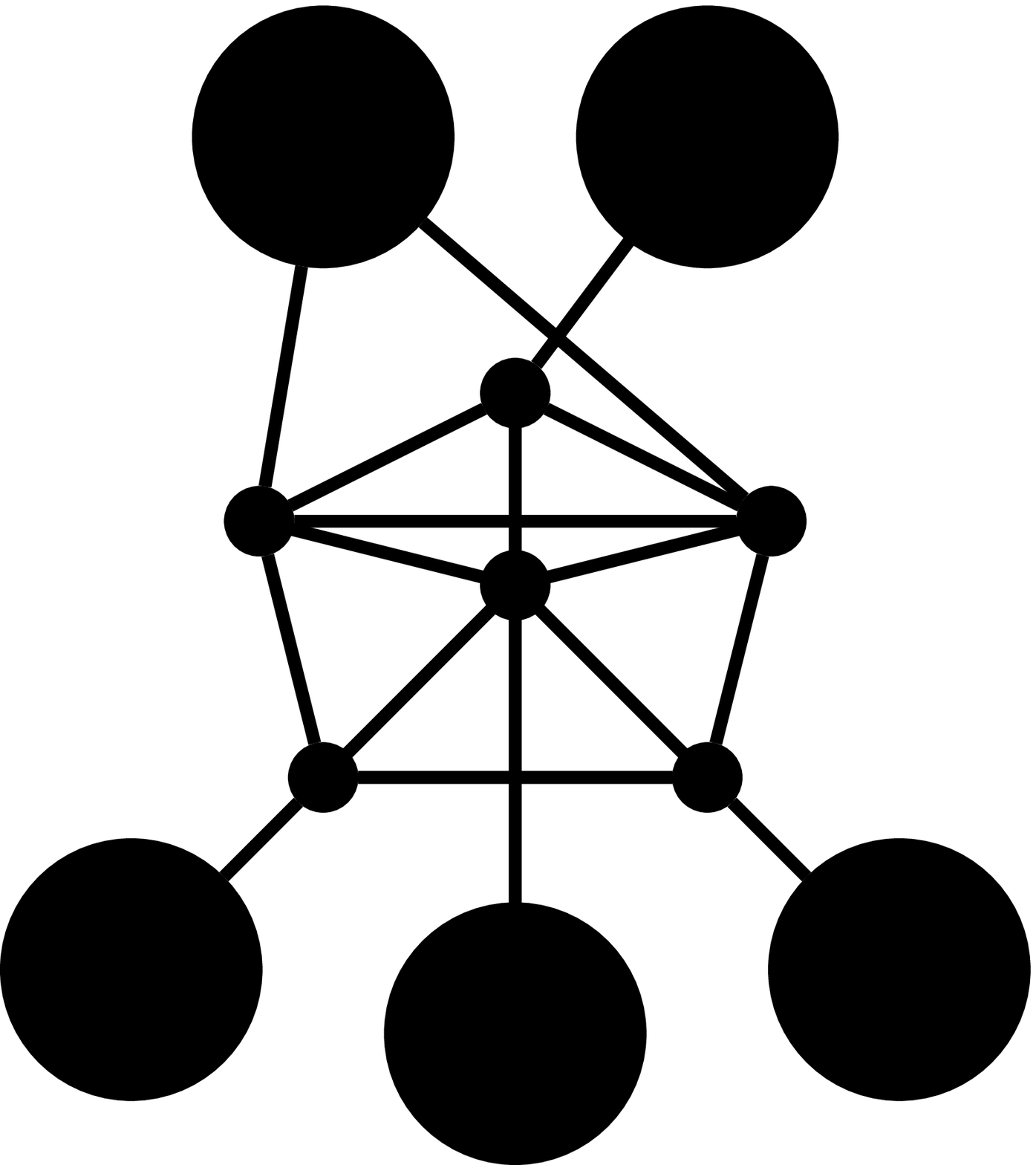} &
\includegraphics[scale=0.15]{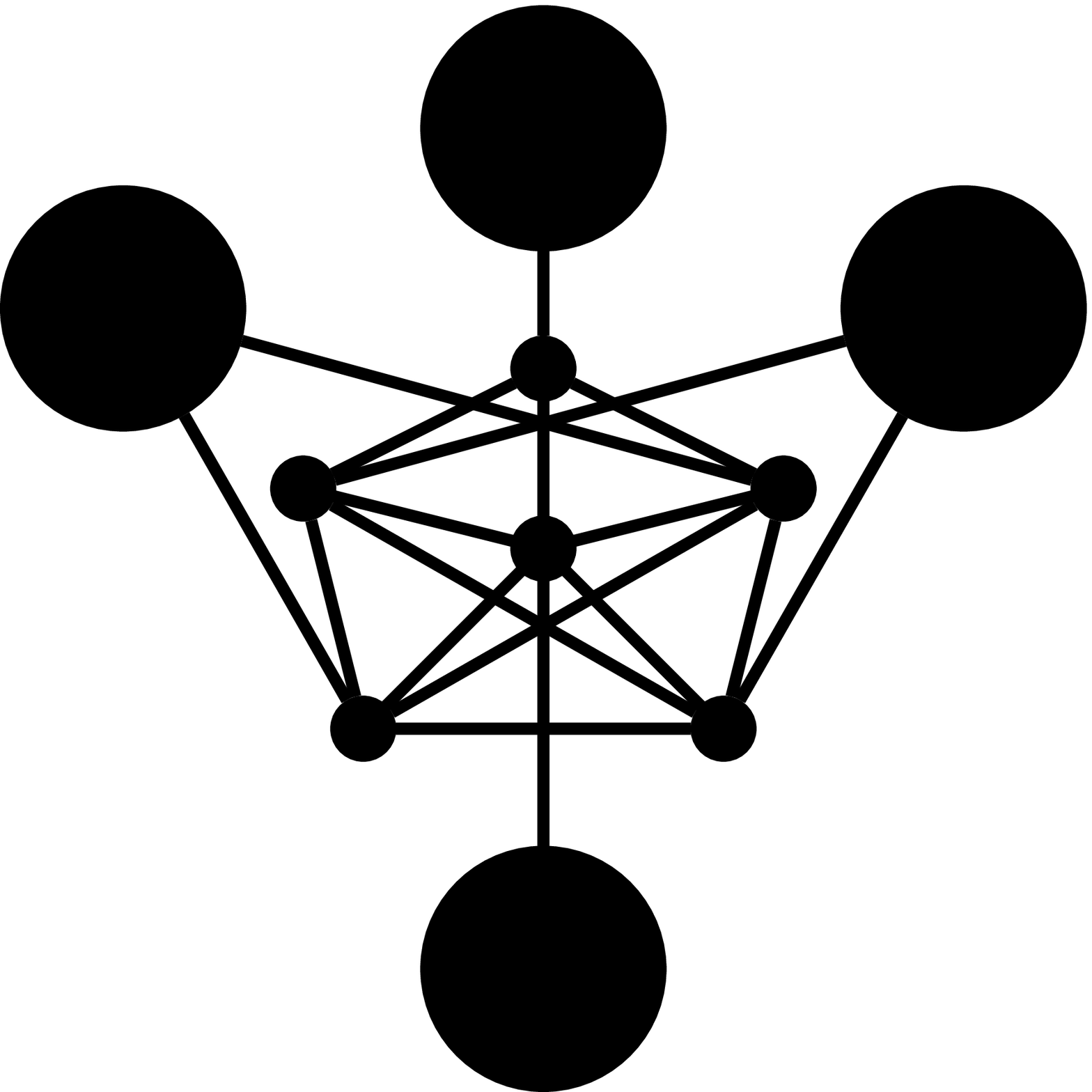} \\
$\Ho_{6,1}^1$ & $\Ho_{6,1}^2$ & $\Ho_{6,1}^3$ \\
    && \\
    && \\
\includegraphics[scale=0.15]{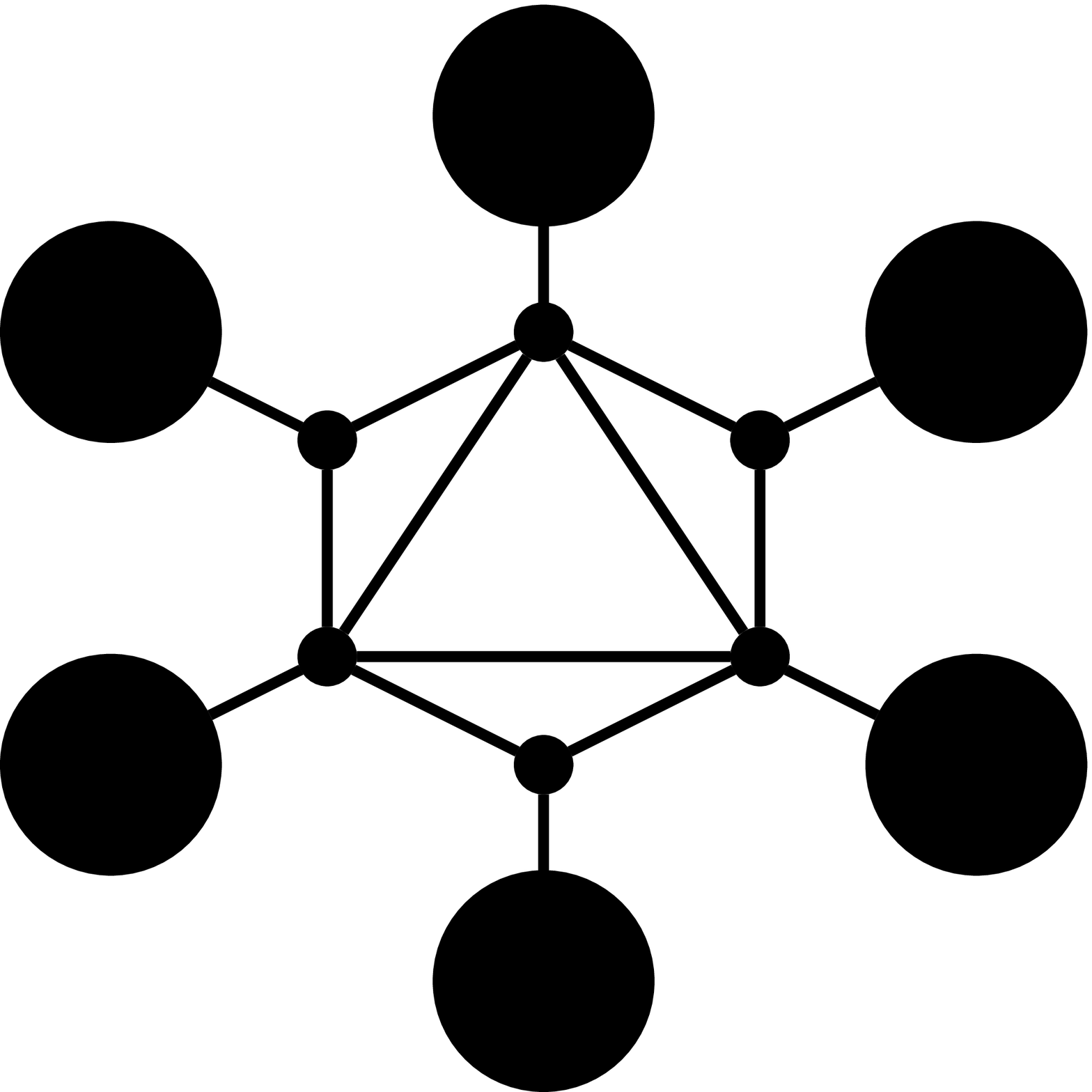} &
\includegraphics[scale=0.15]{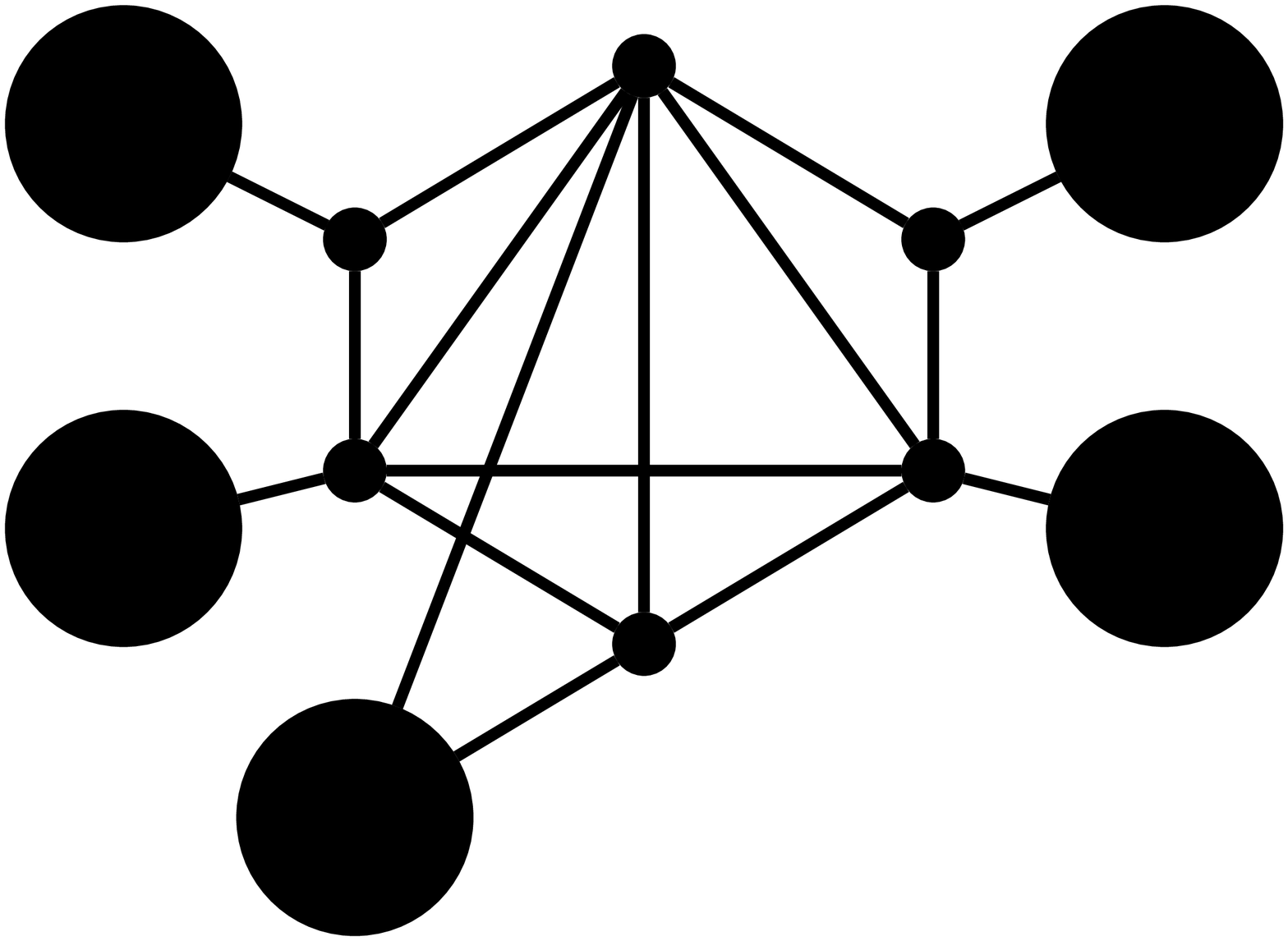} & 
\includegraphics[scale=0.15]{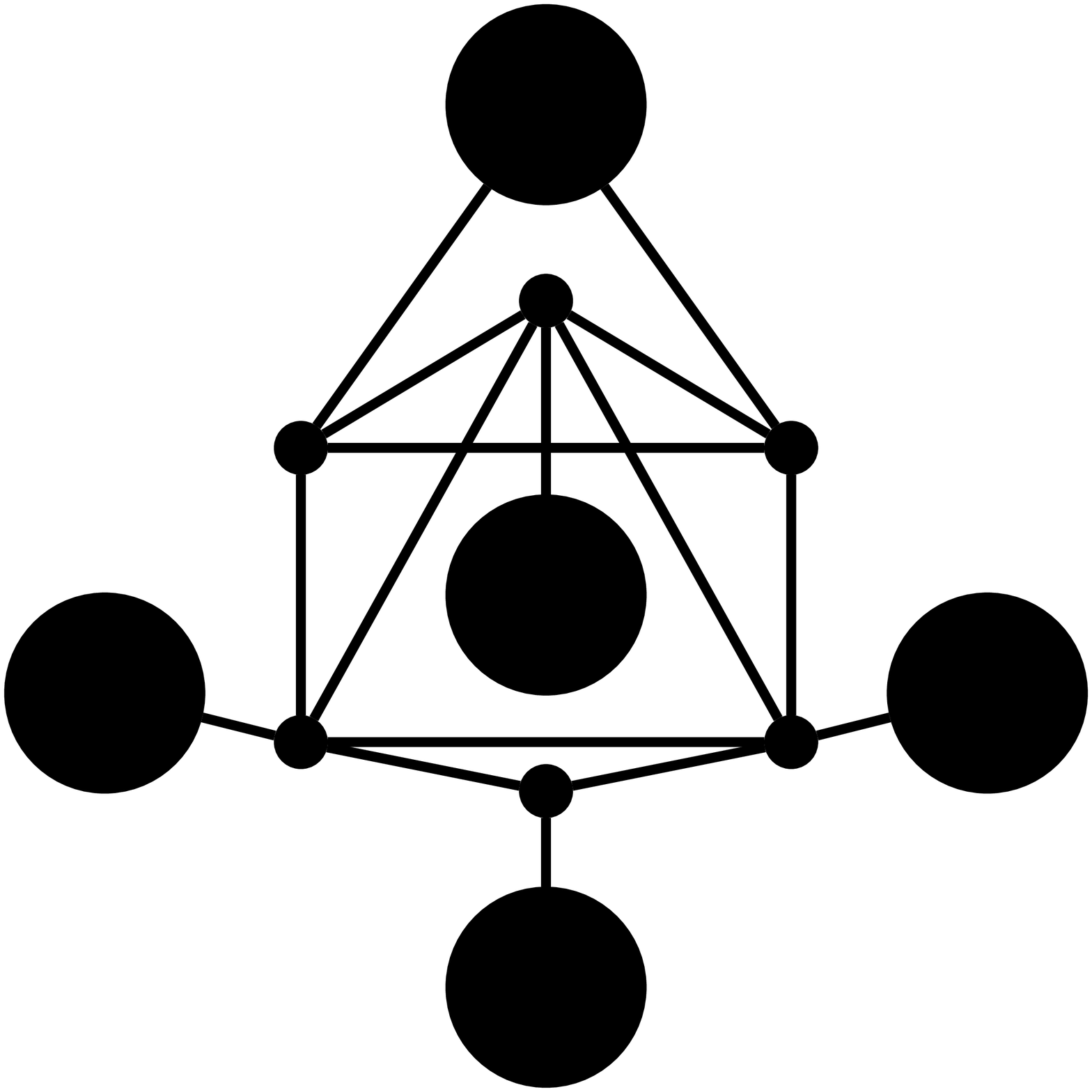} \\
$\Ho_{6,2}^1$ & $\Ho_{6,2}^2$ & $\Ho_{6,2}^3$ \\ 
    && \\
    && \\
\includegraphics[scale=0.15]{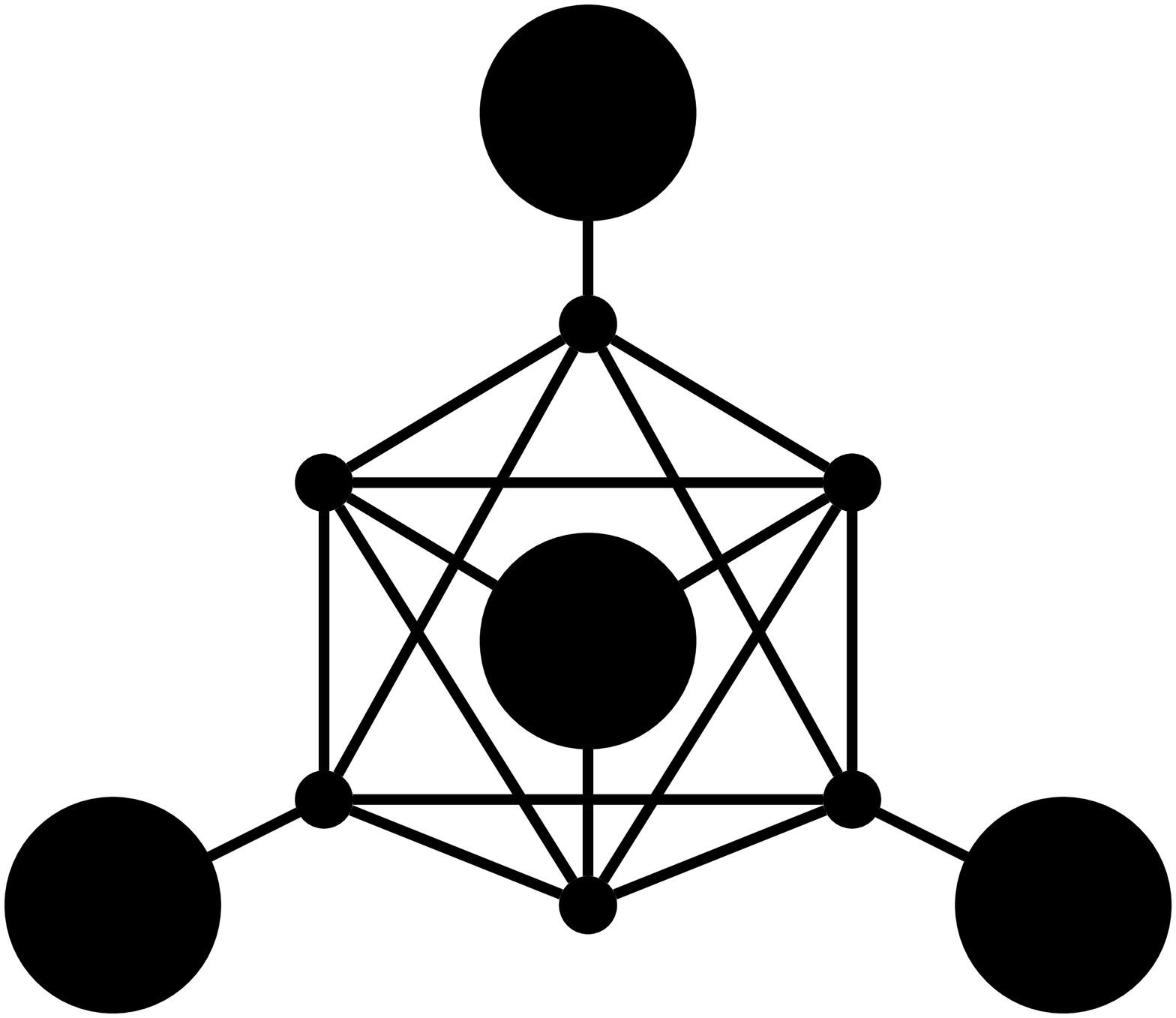} &
\includegraphics[scale=0.15]{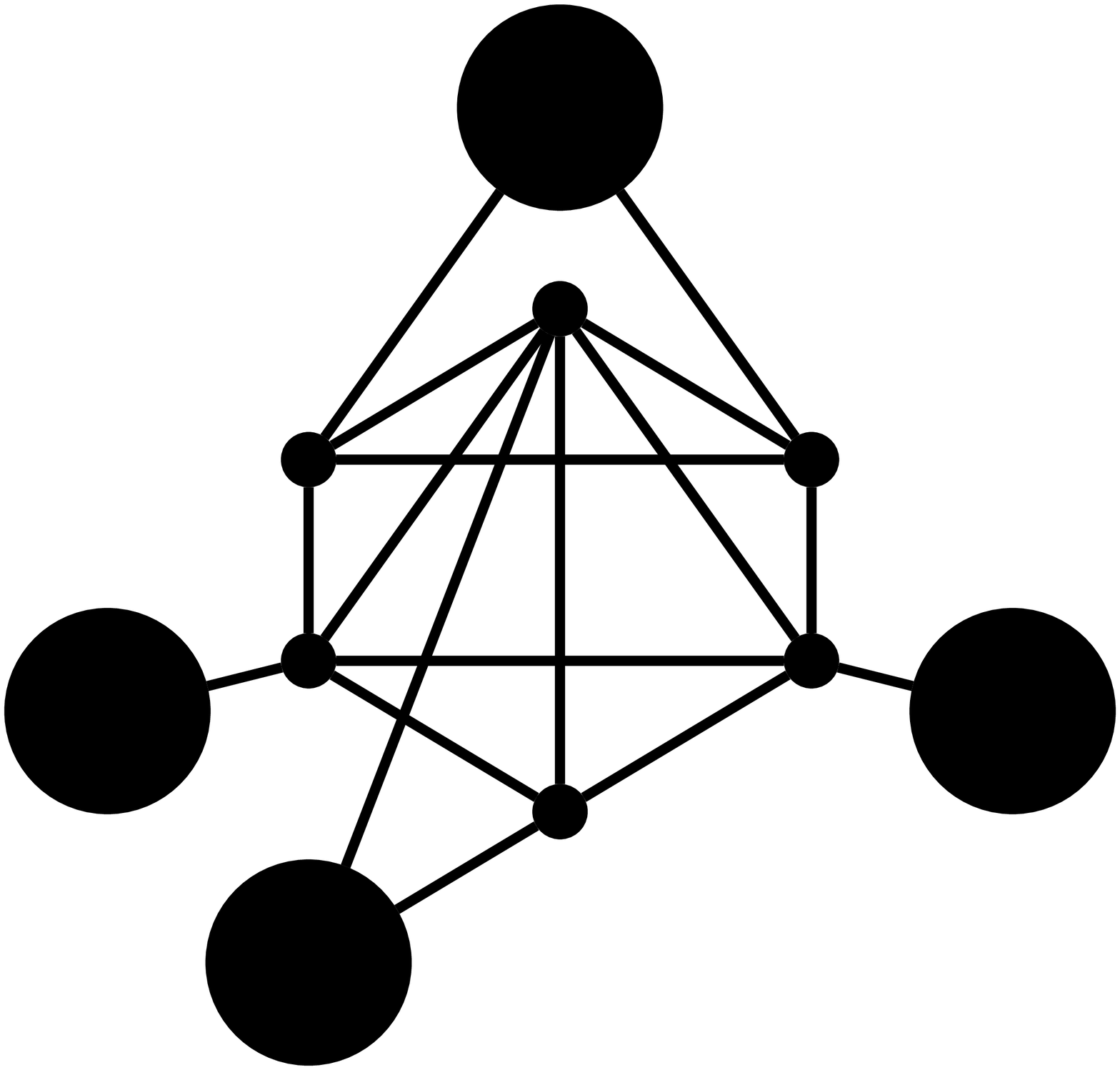} &
 \\
$\Ho_{6,2}^4$ & $\Ho_{6,2}^5$ &  \\
    && \\
    && \\
\includegraphics[scale=0.15]{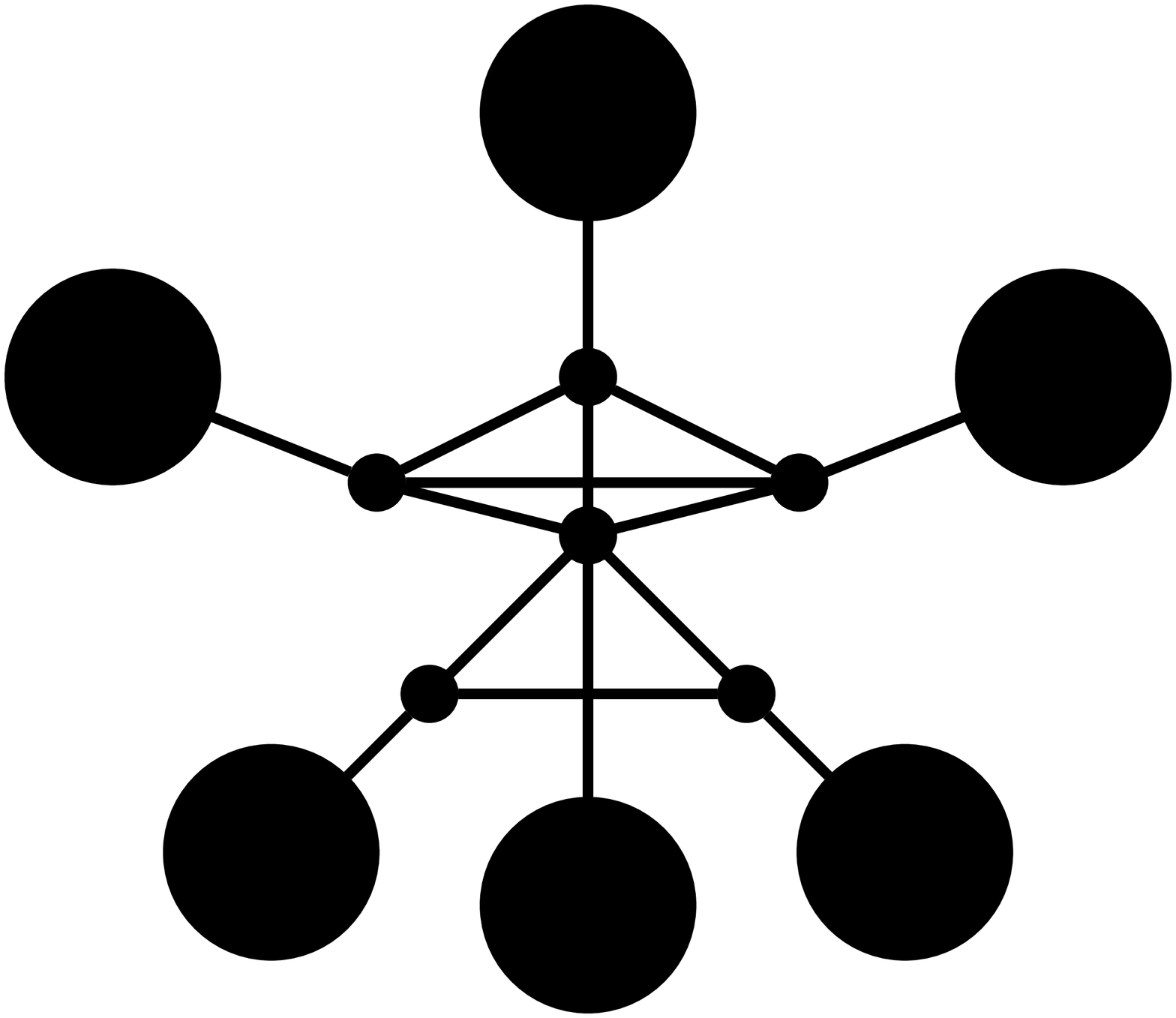} &
\includegraphics[scale=0.15]{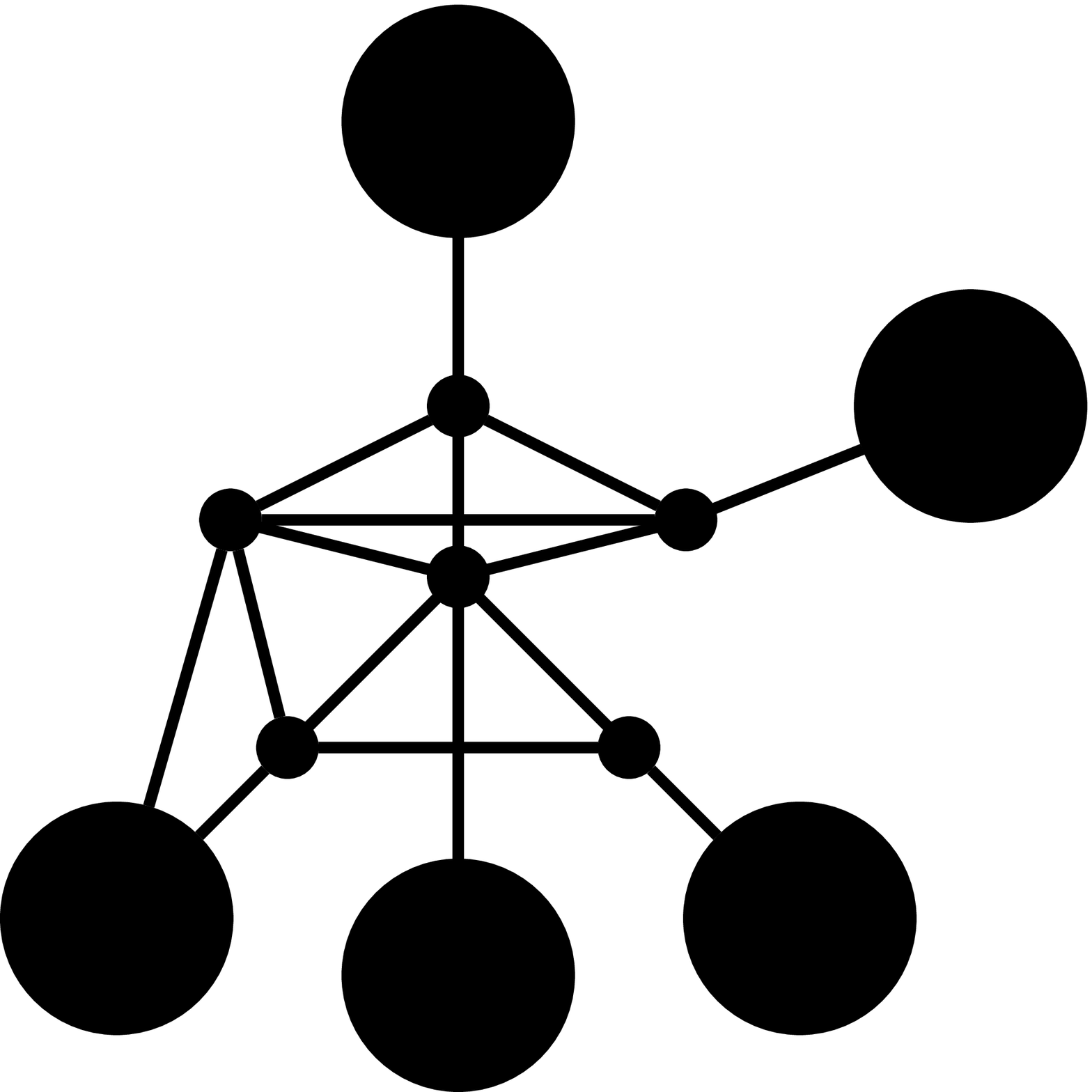} &
\includegraphics[scale=0.15]{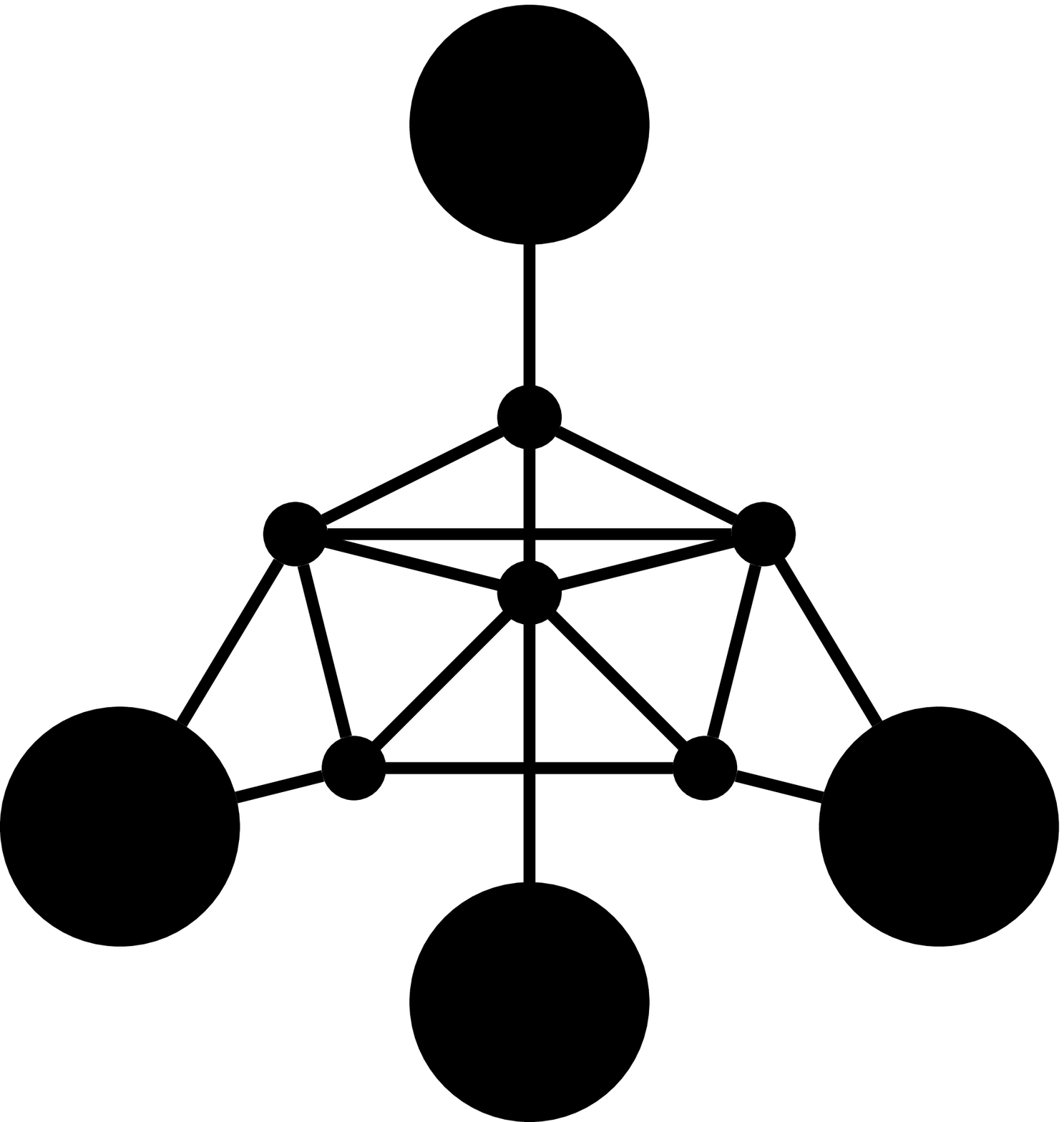} \\
$\Ho_{6,3}^1$ & $\Ho_{6,3}^2$ & $\Ho_{6,3}^3$ \\
    && \\
\end{tabular}
\end{center}
\caption{The fat 
$(-1-\tau)$-ir\-re\-duc\-i\-ble 
Hoffman graphs with $6$ slim vertices}
\label{fig:FHG-6}
\end{figure}

\begin{lemma}\label{lm:FI-1-2}
Let $\Ho$ be a fat 
indecomposable Hoffman graph 
with smallest eigenvalue at least $-1-\tau$. 
If the number of slim vertices of $\Ho$ is at most two, 
then $\Ho$ is isomorphic to one of 
$\Ho_{{\rm I}}$, $\Ho_{{\rm II}}$, $\Ho_{{\rm III}}$, 
$\Ho_{{\rm IV}}$, 
$\Ho_{{\rm XVI}}$, and $\Ho_{{\rm XVII}}$. 
\end{lemma}

\begin{proof}
Straightforward. 
\end{proof}

\begin{lemma}\label{lm:FI-3-}
Let $\Ho$ be a fat 
$(-1-\tau)$-ir\-re\-duc\-i\-ble Hoffman graph. 
If $\s(\Ho)$ is 
isomorphic to $\s_{i,j}$ 
in Figure \ref{fig:004}, 
then 
$\Ho$ is isomorphic to $\Ho_{i,j}^k$ for some $k$
in Figures~\ref{fig:FHG-4}, \ref{fig:FHG-5}, and \ref{fig:FHG-6}. 
\end{lemma}

\begin{proof}
By Lemma~\ref{lm:004}, every slim vertex has exactly
one fat neighbor. 
It is then straightforward to establish the lemma. 
\end{proof}

\begin{theorem}\label{thm:28}
Let $\Ho$ be a fat 
$(-1-\tau)$-ir\-re\-duc\-i\-ble Hoffman graph. 
Then $\Ho$ is isomorphic to one of 
$\Ho_{{\rm I}}$, $\Ho_{{\rm II}}$, $\Ho_{{\rm III}}$, 
$\Ho_{{\rm XVI}}$, $\Ho_{{\rm XVII}}$, and 
the $32$ Hoffman graphs 
given in Figures \ref{fig:FHG-4}, \ref{fig:FHG-5}, and \ref{fig:FHG-6}.
\end{theorem}

\begin{proof}
By Corollary~\ref{thm:001-2}, the special graph $\s(\Ho)$ of
$\Ho$ is isomorphic to 
$\mathcal{Q}_{0,0,1}$, $\mathcal{Q}_{1,0,1}$, $\mathcal{Q}_{0,1,1}$, 
or one of the $15$ 
edge-signed graphs in Figure~\ref{fig:004}. 
If the number of slim vertices of $\Ho$ is at most two, 
then the statement holds 
by Lemma \ref{lm:FI-1-2} and Example \ref{ex:H-4}. 
If the number of slim vertices of $\Ho$ is at least three, 
then the statement holds 
by Lemma \ref{lm:FI-3-}. 
\end{proof}

\begin{theorem}
Let $\mathcal{H}$ be the set of isomorphism classes of
the maximal members of the $37$ fat 
$(-1-\tau)$-ir\-re\-duc\-i\-ble Hoffman graphs given 
in Theorem~\ref{thm:28}, 
with respect to taking 
induced Hoffman subgraphs. 
More precisely, $\mathcal{H}$
is the set of isomorphism classes of 
$\Ho_{{\rm XVI}}$, $\Ho_{{\rm XVII}}$, 
$\Ho_{4,1}^1$, $\Ho_{4,3}^2$, 
$\Ho_{5,2}^1$, $\Ho_{5,3}^1$, 
$\Ho_{5,6}^1$, 
and the $11$ Hoffman graphs in Figure~\ref{fig:FHG-6}. 
Then every fat Hoffman graph with smallest eigenvalue
at least 
$-1-\tau$ 
is an $\mathcal{H}$-line graph.
\end{theorem}

\begin{proof}
It suffices to show that every fat indecomposable
Hoffman graph $\Ho$ with smallest eigenvalue
at least $-1-\tau$ is an $\mathcal{H}$-line graph. 
First suppose that some slim vertex of $\Ho$ has two fat neighbors.
Then by Lemma~\ref{lm:004}, 
$\Ho$ has at most two slim vertices, and
by Lemma~\ref{lm:FI-1-2},
$\Ho$ is isomorphic to one of 
$\Ho_{{\rm I}}$, $\Ho_{{\rm II}}$, $\Ho_{{\rm III}}$, 
$\Ho_{{\rm IV}}$, $\Ho_{{\rm XVI}}$, and $\Ho_{{\rm XVII}}$. 
Since $\Ho_{{\rm I}}$ and $\Ho_{{\rm II}}$ are
induced Hoffman subgraphs of $\Ho_{{\rm XVI}}$,
and $\Ho_{{\rm III}}$ is an
induced Hoffman subgraph of $\Ho_{{\rm XVII}}$,
they are $\mathcal{H}$-line graphs. 
Note that $\Ho_{{\rm IV}}$ is also an
$\mathcal{H}$-line graph since Example~\ref{ex:H-4}
shows that $\Ho_{{\rm IV}}$ is an induced Hoffman subgraph of
a Hoffman graph having a decomposition 
into two induced Hoffman subgraphs
isomorphic to $\Ho_{{\rm I}}$.
Thus the result holds in this case. 

Next suppose that every slim vertex of $\Ho$ has exactly
one fat neighbor. Then by Theorem~\ref{thm:001-1} (ii),
$\s(\Ho)$ is isomorphic to $\mathcal{Q}_{p,q,r}$ 
for some non-neg\-a\-tive integers $p,q,r$ with 
$p+q\leq r$ 
or one of the $15$ edge-signed graphs in 
Figure~\ref{fig:004}. 
In the former case, $\Ho$ is an 
$\mathcal{H}$-line graph by Lemma~\ref{p:4-001}. 
In the latter case, Lemma \ref{lm:FI-3-} implies that $\Ho$ is an 
$\mathcal{H}$-line graph since $\mathcal{H}$ contains 
all the maximal members of the isomorphism classes of 
Hoffman graphs $\Ho_{i,j}^k$. 
\end{proof}


\end{document}